\documentclass[12pt]{amsart}

\usepackage[utf8]{inputenc}
\usepackage{amsmath,amssymb,amsfonts,amsthm}
\usepackage{mathtools} 
\usepackage{enumitem}  
\usepackage[
    margin=1in
]{geometry}
\usepackage[normalem]{ulem}

\usepackage[backend=biber, 
	giveninits=true,
	hyperref=true, 
	maxcitenames=3,
	maxbibnames=100,
	block=none,
	isbn=false]{biblatex}	
	\AtEveryBibitem{%
		\clearfield{mrclass}%
		\clearfield{series}%
		\clearfield{mrnumber}%
		\clearfield{mrreviewer}%
		\clearfield{doi}%
		\ifentrytype{online}{}{
			\clearfield{url}
			\clearfield{note}
		}
	}
\usepackage{graphicx}
\usepackage{tikz-cd}
\usepackage[unicode,psdextra]{hyperref}
\usepackage[all]{xy}
\usepackage{cleveref}

\newcommand{\z}{^{\scalebox{.7}{$\scriptstyle (0)$}}} 
\newcommand{\inv}{^{\scalebox{.7}{$\scriptstyle -1$}}}
\newcommand{\comp}{^{\scalebox{.7}{$\scriptstyle (2)$}}}

\newcommand{\restrict}[2]{{
		#1%
		\left.%
		\kern-\nulldelimiterspace
		\vphantom{|}
		\right|_{#2}
}}

\DeclareMathOperator{\supp}{supp}
\DeclareMathOperator{\Prim}{Prim}
\DeclareMathOperator{\dr}{dr}
\DeclareMathOperator{\dad}{DAD}

\newcommand{\etale}{{\'e}tale}
\newcommand{\LCH}{locally compact, Hausdorff}
\newcommand{\cpc}{c.p.c.}

\newcommand{\red}{{\mathrm{r}}}
\newcommand{\ess}{{\mathrm{ess}}}
\newcommand{\nuc}{{\mathrm{nuc}}}

\newcommand{\dd}{\mathrm{d}}
\newcommand{\id}{\mathrm{id}}

\newcommand{\NN}{\mathbf{N}}
\newcommand{\RR}{\mathbf{R}}
\newcommand{\ZZ}{\mathbf{Z}}
\newcommand{\CC}{\mathbf{C}}
\newcommand{\TT}{\mathbf{T}}

\newcommand{\Ff}{\mathcal{F}}
\newcommand{\Aa}{\mathcal{A}}
\newcommand{\Rr}{\mathcal{R}}
\newcommand{\Hh}{\mathcal{H}}
\newcommand{\Dd}{\mathcal{D}}
\newcommand{\Cst}{\mathrm{C}^*}

\newcommand{\cX}{X_{+}}
\newcommand{\cG}{\widetilde{G}} 
\newcommand{\cl}[2][]{{\overline{{#2}}}^{#1}}
\newcommand{\net}[2]{\{#1\}_{#2}}

\theoremstyle{plain} 
    \newtheorem{thm}{Theorem}[section]
    \newtheorem{prop}[thm]{Proposition}
    
    \newtheorem{lemma}[thm]{Lemma}
    \newtheorem{corollary}[thm]{Corollary}

\numberwithin{equation}{section}

\theoremstyle{remark}
    \newtheorem{remark}[thm]{Remark}

\theoremstyle{definition} 
    \newtheorem{example}[thm]{Example}
    \newtheorem{defn}[thm]{Definition}

\title[Nuclear dimension of twisted groupoid $\Cst$-algebras]{%
Alexandrov groupoids and the
nuclear dimension of twisted groupoid  $\Cst$-algebras
}

\author[K.\ Courtney]{Kristin Courtney}
    \address{%
    Department of Mathematics and Computer Science,
    University of Southern Denmark, 
    Campusvej 55, 
    DK-5230 Odense M, 
    Denmark
    }%
    \email{kcourtney@imada.sdu.dk}
    
\author[A.\ Duwenig]{Anna Duwenig}
\address{Department of Mathematics, KU Leuven, Celestijnenlaan 200B, 3001 Leuven, Belgium}
\email{anna.duwenig@kuleuven.be}
\author[M.\ C.\ Georgescu]{Magdalena C.\ Georgescu}
    \address{Toronto, Ontario, Canada}
    \email{mcgeorgescu@gmail.com}
    \author[A.\ an Huef]{Astrid an Huef}
    \address{School of Mathematics and Statistics, Victoria University of Wellington, P.O.\ Box 600, Wellington 6140, Aotearoa New Zealand}
    \email{astrid.anhuef@vuw.ac.nz}
\author[M.\ G.\ Viola]{Maria Grazia Viola}
    \address{Lakehead University, Orillia ON L3V 0B9, Canada, and University of Toronto, Bahen Centre, 40 St. George Street, Room 6290, Toronto, ON M5S 2E4, Canada}
    \email{mviola@lakeheadu.ca}



\thanks{%
This research was initiated during  the project-oriented workshop ``Women in Operator Algebras II'' funded and hosted by the Banff International Research Station. It was further supported by the Deutsche Forschungsgemeinschaft (DFG, German Research Foundation) under Germany's Excellence Strategy -- EXC 2044 -- 390685587, Mathematics M\"unster -- Dynamics -- Geometry -- Structure; the Deutsche Forschungsgemeinschaft (DFG, German Research Foundation) – Project-ID 427320536 – SFB 1442; and ERC Advanced Grant 834267 - AMAREC. 
Astrid an~Huef was supported by  the  Marsden Fund of the Royal Society of New Zealand (grant number 21-VUW-156). Maria Grazia Viola was supported by the National Sciences and Engineering Research Council of Canada (RGPIN-386687).  
We thank the referees for helpful comments.
}%

\date{\today}

\begin{document}

\maketitle

\begin{abstract}
We consider a twist $E$ over an \etale\ groupoid $G$. When $G$ is principal, we  prove that the nuclear dimension of the reduced twisted groupoid $\Cst$-algebra is bounded by a number  depending on the dynamic asymptotic dimension of $G$ and the topological covering dimension of its unit space. This generalizes an analogous theorem by \citeauthor{GWY:2017:DAD} for the  $\Cst$-algebra of $G$. Our proof uses  a reduction to the  unital case where $G$ has compact unit space, via a construction of ``groupoid unitizations''  $\widetilde{G}$ and $\widetilde{E}$ of $G$ and  $E$ such that  $\widetilde{E}$ is a twist over $\widetilde{G}$. The construction of $\widetilde G$ is for r-discrete (hence 
for \etale)   groupoids $G$ which are not necessarily principal. When $G$ is \etale, the dynamic asymptotic dimension of $G$ and $\cG$ coincide.
We show that the minimal unitizations of the full and reduced twisted groupoid $\Cst$-algebras of the twist over $G$ are isomorphic to the twisted groupoid $\Cst$-algebras of the twist over $\widetilde{G}$.
We apply our result about the nuclear dimension of the twisted groupoid $\Cst$-algebra to obtain a  similar  bound on the nuclear dimension of the $\Cst$-algebra of an \etale\ groupoid with closed orbits and  abelian stability subgroups that vary continuously.
\end{abstract}

\section{Introduction}

Noncommutative generalizations of topological covering dimension have been employed in numerous contexts to capture 
rigidity phenomena. One that has received considerable attention in recent years in the study of 
$\Cst$-algebras is a refinement of nuclearity called nuclear dimension \cite{WinterZacharias:nuclear}, which provided the necessary additional 
rigidity for the completion of the classification program: all simple, separable, unital $\Cst$-algebras with finite nuclear dimension which satisfy a universal coefficient theorem (UCT) are classified by their K-theoretic and tracial data \cite{Tikuisis-White-Winter:QD}.  While finite nuclear dimension is a prerequisite of the classification result, simplicity forces the nuclear dimension of a classifiable $\Cst$-algebra to be either 0 or 1 \cite{CETWW}.

In recent years, the attention of $\Cst$-algebraists has  turned towards both determining whether a given $\Cst$-algebra is classifiable  (e.g., \cite{AAFGJST22,   EG16, GGKN,  GGKNV, GGNV, KeSz20}) and expanding classification to
the non-simple setting (e.g., \cite{BGSW22, W+20, EGM19, EM18, Evi22, RST15}). Results and techniques for  non-simple $\Cst$-algebras 
remain ad hoc, and a wealth of regularity properties 
are required for progress. Because of its permanence properties under various constructions, nuclear dimension remains an invariant of great interest  in this context.  A promising approach is the emerging correspondence between nuclear dimension and other noncommutative generalizations of topological covering dimension, including tower dimension \cite{Kerr:dim},  dynamic asymptotic dimension \cite{GWY:2017:DAD}, and diagonal dimension \cite{LLW}.

Asymptotic dimension was originally introduced by Gromov
as a coarse geometric analogue of covering dimension in topology \cite{Gro93}. 
In \cite{GWY:2017:DAD}, Guentner, Willett, and Yu generalized the notion of asymptotic dimension of a group to \LCH\ and 
\etale\  groupoids. 
 Groupoids and their associated $\Cst$-algebras have long been studied by  $\Cst$-algebraists  
 because they provide a rich class of natural and useful examples, and, more recently,
 because they have been shown to have deep connections to lingering problems and future tracks in  the aforementioned classification program in $\Cst$-algebras.

A major remaining problem of the classification program is whether 
the UCT assumption is redundant. 
 A remarkable breakthrough in \cite{BL:Cartan-UCT-I},  building on \cite{Tu99} and \cite{Kum:Diags, Renault:Cartan},  says that if a nuclear $\Cst$-algebra is isomorphic to a twisted groupoid $\Cst$-algebra, then it automatically satisfies the UCT.   Moreover, by recent work of \citeauthor{Li:Classif-Cartan}, any classifiable $\Cst$-algebra is isomorphic to a twisted groupoid $\Cst$-algebra  \cite{Li:Classif-Cartan}.  
 This has contributed to a flurry of research in twisted groupoid $\Cst$-algebras (e.g., \cite{Austad-Ortega, Bonicke, BFRP:Twisted, KLS, Li-Menger}), and it is hence of significant interest to detect, at the groupoid level, when an associated $\Cst$-algebra is classifiable.

For most of the assumptions, this has been explored:
the twists $E$ that give rise to classifiable $\Cst$-algebras are over second-countable, \LCH\ and \etale\ groupoids $G$ which satisfy an aperiodicity condition called effectiveness, and the associated reduced twisted groupoid $\Cst$-algebra $\Cst_\red(E;G)$ is simple if and only if $G$ is minimal  \cite{Armstrong:twistedgpds}. If $\Cst_{\red}(E;G)$ is nuclear, then we know from \cite{BL:Cartan-UCT-I, Tu99} that it automatically satisfies the UCT. A missing piece is conditions that translate to finite nuclear dimension (or, in the presence of nuclearity, other equivalent regularity properties of the $\Cst$-algebra; for example, see~\cite{CETWW}). Such conditions will also propel developments in non-simple classification.

 In Theorem~8.6 of \cite{GWY:2017:DAD}, Guentner, Willett and Yu showed that the nuclear dimension of the reduced $\Cst$-algebra of a second-countable, \LCH, principal and \etale\ groupoid $G$ is bounded by a number depending on the dynamic asymptotic dimension of the groupoid and the topological covering dimension of the  unit space of the groupoid. Many interesting $\Cst$-algebras, such as separable AF algebras, fit into this picture.  However, as far as we currently know, twists may be required to pick up all classifiable $\Cst$-algebras and there exist twisted groupoid $\Cst$-algebras that cannot be realized as non-twisted groupoid $\Cst$-algebras  \cite{BS21}. Thus, we want to detect finite nuclear dimension for a twisted groupoid $\Cst$-algebra.
 
 Our main result (Theorem \ref{holy grail}) shows that if $E$ is a twist over a second-countable, \LCH, principal and \etale\ groupoid $G$, then the nuclear dimension of $C_{\red}^\ast(E;G)$ is bounded  by  $(N+1)(d+1)-1$ where $N$ is the topological covering dimension of the unit space of $G$ and $d$ is the dynamic asymptotic dimension of $G$. 
 To prove this, we show that this $\Cst$-algebra satisfies a colored version of local subhomogeneity that appears implicitly in various proofs in the literature, including
\cite[Theorem~8.6]{GWY:2017:DAD} and \cite[Theorem~6.2]{Kerr:dim}.
 We show in \Cref{prop: colored loc subhom}, that in all these proofs it is this colored local subhomogeneity that determines the nuclear dimension of the associated $\Cst$-algebra.
 To prove that our twisted groupoid $\Cst$-algebras satisfy the assumptions of \Cref{prop: colored loc subhom}, we take inspiration from \cite{GWY:2017:DAD}  but need to take into account the  extra structure coming from the twist. 
 The subgroupoids that feature in the definition of dynamic asymptotic dimension 
 give rise to twists whose  $\Cst$-algebras have primitive ideal spaces that may not be Hausdorff, and hence finding bounds on the  topological dimension of these spaces is delicate (see \Cref{subgroupoid}). Crucial for this analysis is a description of the primitive ideal space of twisted groupoid $\Cst$-algebras by Clark and the fourth author in \cite{Clark-anHuef-RepTh}, and an understanding of the spaces of certain regular representations from work of Muhly and Williams in \cite{Muhly-Williams-CtsTr2}.  Building on these results, we show in Proposition~\ref{subgroupoid} that the twisted-groupoid $\Cst$-algebras of these subgroupoids have a recursive subhomogeneous decomposition, as studied in \cite{Phillips:recursive}, with  topological dimension at most that of their unit spaces.

 We also develop the Alexandrov compactifications of  an \'etale groupoid $G$  and of a twist $E$ over $G$; these   ``compactifications''  produce a twist $\widetilde{E}$ over $\widetilde{G}$ that witnesses the unitization of  the twisted groupoid $\Cst$-algebra at the groupoid level. 
 Thus the reduced (respectively, full) twisted groupoid $\Cst$-algebra
of the  twist $\widetilde{E}$  over $\widetilde{G}$
is isomorphic to the minimal unitization of the reduced (respectively, full) twisted groupoid $\Cst$-algebra of the twist over $G$.
We then show that  the 
dynamic
asymptotic dimension of the unit space of $G$ coincides with that of  $\widetilde{G}$, and this allows us to reduce the proof of Theorem~\ref{holy grail} to the unital case.

A surprising application of our main theorem is  a bound on the nuclear dimension of the $\Cst$-algebra  of  an \'etale  groupoid  $G$ with closed orbits but  potentially large abelian isotropy subgroups, provided these subgroups vary continuously  (\Cref{non-principal groupoids}).  
Here  $\Cst_\red(G)$ is isomorphic to a twisted groupoid $\Cst$-algebra over  a principal groupoid by \cite{Clark-anHuef-RepTh}.  This principal groupoid  is a  transformation groupoid $\widehat{\Aa}\rtimes\Rr$ obtained from an action of  the quotient groupoid $\Rr=G/\Aa$ of $G$ by the isotropy subgroupoid $\Aa$   on the spectrum $\widehat{\Aa}$ of the abelian $\Cst$-algebra of $\Aa$.  We show that if $\Rr$ has finite dynamic asymptotic dimension at most $d$, then so does $\widehat{\Aa}\rtimes\Rr$; thus  our \Cref{holy grail}  applies, provided a bound on the topological covering dimension of the unit space $\widehat{\Aa}$ can be computed.

We illustrate \Cref{non-principal groupoids} with two examples. The first, \Cref{ex: transformation}, is a transformation group whose $\Cst$-algebra has continuous trace but is not AF;
here we find that the nuclear dimension is $1$. The second,  \Cref{ex-CaH}, is the groupoid   of a directed graph $E$ whose $\Cst$-algebra  is AF embeddable. Since $\Cst(E)$ does not have a purely infinite ideal, it is outside the scope of the results on  nuclear dimension of graph $\Cst$-algebras in \cite{RST15}.    Moreover, since $E$ has return paths,  by \Cref{lemma-directed-graph-groupoids} the dynamic asymptotic dimension of the graph groupoid is  $\infty$, and hence is a poor predictor of the nuclear dimension of its $\Cst$-algebra. But in \Cref{non-principal groupoids} we consider a  groupoid 
modulo its isotropy groupoid, and 
this gives more information about the $\Cst$-algebra: we find that the nuclear dimension of $\Cst(E)$ is actually $1$.

\subsubsection*{Structure of the paper}
In \Cref{Preliminaries} we recall the definition of a groupoid twist, exhibit an induced Haar system on a twist over an \etale\ groupoid
and define the twisted groupoid $\Cst$-algebra. We also recall the definitions of nuclear dimension, decomposition rank and dynamic asymptotic dimension. In \Cref{unitization} we introduce the Alexandrov groupoid $\widetilde{G}$ of an r-discrete groupoid $G$, construct a twist over $\widetilde{G}$ from a twist over $G$, and study the associated $\Cst$-algebras. In \Cref{twistedgroupoidCalgebras} we prove our main 
    theorem; in 
particular, for \Cref{holy grail} and \Cref{subgroupoid} we assume that $E$ is a twist over an \etale\ and principal groupoid.  In \Cref{sec-application} we  develop our application to non-principal groupoids and discuss the two examples mentioned above.

\section{Preliminaries}\label{Preliminaries}

\subsection{Groupoids}

Let $E$ be a  \LCH\ groupoid. We denote by $E\z$ the unit space of $E$. The range and source maps $r,s\colon E\to E\z  $ are given by $r(e)=ee\inv$ and $s(e)=e\inv e$, respectively. The set of composable pairs $\{(d,e) : s(d)=r(e)\}$ is denoted by $E\comp $.  Then  $E$ is \emph{principal} if the map $e\mapsto (r(e),s(e))$ is injective. We say that $E$ is  {\em r-discrete} if its unit space $E\z$ is open and that $E$ is \emph{\etale} if the range map is a local homeomorphism. Notice that every \etale\ groupoid is  r-discrete.  For a subset $X\subset E$ we write $\langle X \rangle$ for the subgroupoid of $E$ generated by $X$. 

Let $W\subset E\z  $. We define the \emph{restriction of $E$ to $W$} to be 
\[
\restrict{E}{W}=\{e\in E : s(e), r(e)\in W\}.
\]
Note that $\restrict{E}{W}$ is an algebraic groupoid  with unit space $W$.  If, for example, $W$ is locally closed in $E\z  $, then $\restrict{E}{W}$ is a \LCH\ groupoid (see \Cref{restriction of twist non-etale} below).

The \emph{saturation} of $W$ is $r(s\inv (W))=s(r\inv (W))$, and if $W=r(s\inv (W))$ then we say that $W$ is \emph{saturated}. If $x\in E\z  $, we call the saturation of $\{x\}$ the \emph{orbit} of $x$ and denote it by $[x]$.  We also  set $xE\coloneqq\{e\in E : r(e)=x\}$ and $Ex\coloneqq\{e\in E : s(e)=x\}$. The group $xE\cap Ex$ is called the \emph{stability subgroup at $x$} and is also known as the \emph{isotropy subgroup at $x$}.

Now suppose that $E$ has a left Haar system $\sigma=\{\sigma^x : x\in E\z\}$; if $E$ is \etale, then we choose the Haar system of counting measures. We  recall the definition of the usual full and reduced $\Cst$-algebras of $E$. Let  $C_c(E)$ be the vector space of continuous and compactly supported $\CC$-valued functions on $E$, equipped  with convolution and involution given for $e\in E$ and $f,g\in C_c(E)$ by
\[
f*g(e)=\int_E f(d)g(d\inv e)\,\dd\sigma^{r(e)}(d)\quad\text{and}\quad f^*(e)=\overline{f(e\inv )} .
\]
A $*$-homomorphism $L\colon C_c(E)\to B(H_L)$
into the bounded operators on some Hilbert space $H_L$
is a \emph{representation} if it is $I$-norm bounded, or equivalently, if it is continuous when $C_c(E)$ has the inductive limit topology and  $B(H_L)$ the weak operator topology \cite[Remark~1.46]{Williams:groupoid}. Then the \emph{full $\Cst$-algebra} $\Cst(E,\sigma)$ of $E$ is the completion of $C_c(E)$ in the norm
\[
\|f\|_{\Cst(E,\sigma)}=\sup\{\|L(f)\| : \text{$L$ is a representation of $C_c(E)$}\}.
\]

As usual, we define another measure $\sigma_x$ on $E$ with support in $Ex$ by $\sigma_x(U)=\sigma^x(U\inv )$.
For each $x\in E\z  $, define $L^x\colon C_c(E)\to B(L^2(Ex,\sigma_x))$ for $e\in Ex$ by
 \begin{equation}\label{Lrep}
 \big(L^x(f)\xi\big)(e)=\int_E f(d)\xi(d^{\inv} e)\, \dd\sigma^{r(e)} (d)
 \end{equation}
 Notice that if $\xi\in C_c(Ex) \subset L^2(Ex,\sigma_x)$, then $L^x(f)\xi=f*\xi$.  By \cite[Proposition~1.41]{Williams:groupoid}, $L^x$ is a representation of $C_c(E,\sigma)$  and hence extends to a representation $L^x\colon \Cst(E)\to B(L^2(Ex,\sigma_x))$.  The  \emph{reduced $\Cst$-algebra} $\Cst_\red(E,\sigma)$ of $E$  is the completion of $C_c(E)$ in the norm
 \[\|f\|_{\Cst_\red(E,\sigma)} =\sup\{\|L^x(f)\| : x\in E\z  \}.\]

\subsection{Twists}

There are various definitions of twists in the literature, for example, \cite[Definition~2.4]{Kum:Diags}, \cite[\S2]{Muhly-Williams-CtsTr2} \cite[\S4]{Renault:Cartan}, \cite[Definition~11.1.1]{Sims:gpds} and \cite[Definition~3.1]{Bonicke}, and they are not always consistent.  Here we want to allow twists over possibly non-\etale\ and non-principal groupoids. We start with the definition from \cite{Bonicke}.

 \begin{defn}\label{def:twist2} Let $G$ be a locally compact and Hausdorff groupoid, and regard $G\z \times\TT$ as a trivial group bundle with fibers $\TT$.  A \emph{twist} $(E, \iota, \pi)$ over $G$ consists of a locally compact and Hausdorff groupoid $E$ and groupoid homomorphisms $\iota$, $\pi$ such that 
  \[G\z  \times \TT \stackrel{\iota}{\longrightarrow} E \stackrel{\pi}{\longrightarrow} G\]
is a central groupoid extension, which means that
\begin{enumerate}
    \item\label{item:iota}  $\iota\colon G\z \times\TT\to \pi\inv (G\z)$ is a homeomorphism, where $\pi\inv (G\z)$ has the subspace topology from $E$, and it satisfies
    $\iota(\pi(u),1) = u$ for all $u\in E\z$;
    \item $\pi$ is a continuous, open surjection; and
    \item 
    $\iota(\pi(r(e)), z)e=e\iota(\pi(s(e)), z)$
    for all $e\in E$ and $z\in\TT$.
\end{enumerate}
\end{defn}
The twists in \cite{Muhly-Williams-CtsTr2} and \cite{ Clark-anHuef-RepTh} are  twists in the sense of  \Cref{def:twist2} with the additional assumption that $G$ is principal. Because of Condition~(\ref{item:iota}), we may identify $E\z$ with $G\z$
via the homeomorphism $\pi|_{E\z}$.

Let $(E, \iota, \pi)$ be a twist over a locally compact and Hausdorff groupoid~$G$. For $z\in \TT$ and $e\in E$, define
 \[z\cdot e \coloneqq\iota(r(e), z)e = e\iota(s(e), z).\] 
 This gives  a  continuous action of $\TT$ on $E$ which is free because $\iota$ is injective. If $d, e\in E$ such that $\pi(d)=\pi(e)$, then there exists a unique $z\in\TT$ such that~$d=z\cdot e$. 

We equip the orbit space $E/\TT$ with the quotient topology with respect to the orbit map  $q\colon E\to E/\TT$. Then $q$ is open by \cite[Lemma~ 4.57]{RaWi:Morita}, and so by exactness, $\pi$ induces a homeomorphism of $E/\TT$ onto $G$. Thus,  we may identify $E/\TT$ and $G$, and hence also $q$ and~$\pi$.
Since $E$ is locally compact and Hausdorff, it is completely regular \cite[p.~151]{Pedersen:AnaNow}.  Now~$\TT$ is a Lie group acting freely and, by compactness, properly on the completely regular  space~$E$, and it follows from \cite[Theorem on p.~315]{Palais:slices} that $\pi$ is a locally trivial principal bundle, that is, for very $\alpha\in G$ there exist a neighborhood $U$ of $\alpha$ and  a continuous  map $S\colon U\to E$ such that $\pi\circ S=\id_U$ (see \cite[Proposition~4.65 and Hooptedoodle~4.68]{RaWi:Morita}). It is then routine to check that  $(\beta, z)\mapsto z\cdot S(\beta)$ is a homeomorphism of $U\times\TT$ onto~$\pi\inv (U)$.  

The existence of continuous local sections and the local trivializations are often assumed in the definition of a twist over an \etale\ groupoid (see, for example, \cite{Sims:gpds}), and it then follows that $\pi$ is automatically an open map. Indeed, when $G$ is \etale, \cite[Definition~11.1.1]{Sims:gpds} and \Cref{def:twist2} are equivalent, as outlined in  \cite[Remark~2.6]{Armstrong:twistedgpds}.

If $G$ is r-discrete, then $G\z =E\z$ is open in $G$ and in $E$, and hence $\iota$ is an open map into $E$ as it is a homeomorphism onto an open subset. If $G$ is even \etale, then  the range and  source maps on $E$ are open, and the multiplication map on $E\comp $ is open (see, for example, \cite[Lemma~2.7]{Armstrong:twistedgpds}).

The following \namecref{lem:pi proper} is well known, but we could not find a reference for it.

\begin{lemma}\label{lem:pi proper}
 Let $(E, \iota, \pi)$ be a twist over a \LCH\ groupoid $G$.   Then $\pi$ is a proper map.
\end{lemma}

\begin{proof}
    Let $K\subset G$ be a compact set; we show that any net $\net{e_{i}}{i}$ in $\pi\inv (K)$ has a convergent subnet. Since $G$ is Hausdorff, both $K$ and hence $\pi\inv(K)$ are closed; therefore, if a subnet of $\net{e_{i}}{i}$ converges in $E$, its limit must be in $\pi\inv (K)$.
    
    Let $\gamma_{i}\coloneqq \pi(e_{i})$. Since $\gamma_i\in K$ and $K$ is compact, the net has a convergent subnet; without loss of generality, we may assume that $\net{\gamma_i}{i}$ itself converges to, say, $\gamma$. Let $e\in \pi{\inv} (\gamma)$ be arbitrary. Since $\pi$ is an open surjection, we can use Fell's criterion (\cite[Proposition~1.1]{Williams:groupoid}): there exists a subnet $\net{\gamma_{g(m)}}{m\in M}$ and a net $\net{\tilde{e}_{m}}{m\in M}$, indexed by the same directed set $M$, such that $\tilde{e}_{m}\to e$ in $E$ and $\pi(\tilde{e}_{m})=\gamma_{g(m)}$ for all $m$. By choice of $\gamma_i$, the latter implies that there exist $z_{m}\in \TT$ such that $\tilde{e}_{m} = z_{m}\cdot e_{g(m)}$.
    Since $\TT$ is compact, there exists a subnet $\net{z_{h(n)}}{n\in N}$ of $\net{z_{m}}{m\in M}$ that converges to, say, $z$. Since everything in sight is continuous, we conclude that
    \[\lim_n e_{g(h(n))} = \lim_n \overline{z_{h(n)}}\cdot \tilde{e}_{h(n)}=\overline{z}\cdot e. \]
    In other words, we have found a subnet of $\net{e_{i}}{i}$ that converges. 
\end{proof}

A version of the following \namecref{restriction of twist non-etale} was implicitly used in \cite[Lemma~3.1]{Clark-anHuef-RepTh}. 
Recall that a subset $W$ of a topological space $X$ is  {\em locally closed} if each point in $W$ has an open neighborhood $U$ in $X$ such that $U\cap W$ is closed in $U$.

\begin{lemma}\label{restriction of twist non-etale}
    Let $G$ be a locally compact and  Hausdorff groupoid, and  let $(E, \iota, \pi)$ be a twist over $G$. Let $W$ be a locally closed subset of $G\z  $. Then $\restrict{E}{W}$  and $\restrict{G}{W}$ are locally compact and Hausdorff groupoids, and $(\restrict{E}{W}, \iota|_{W\times\TT}, \pi|_{\restrict{E}{W}})$ is a twist over $\restrict{G}{W}$. If $G$ is r-discrete, \etale\ or principal, then so is $\restrict{G}{W}$.
\end{lemma}

\begin{proof} 
In the special case where  $W$ is either open or closed in $G\z  $, we have that $\restrict{E}{W}=r_E^{\inv}(W)\cap s_E^{\inv}(W)$ is an open (respectively, closed) subset of a locally compact space, hence is locally compact.  Next, suppose that $W$ is locally closed in $G\z  $, so that $W$ is open in its closure $\cl{W}$ by \cite[Lemma~1.25]{Williams:Crossed}. Now $\restrict{E}{\cl{W}}$ is locally compact and $W$ is an open subset of  its unit space $\cl{W}$, and so $\restrict{E}{W}$ is locally compact by the previous argument. Similarly, $\restrict{G}{W}$ is locally compact.
Subsets of Hausdorff spaces are Hausdorff and  restrictions of continuous maps are continuous, and so both $\restrict{E}{W}$ and $\restrict{G}{W}$ are locally compact and   Hausdorff  groupoids with unit spaces $W$. It remains to show that $(\restrict{E}{W}, \iota|_{W\times\TT}, \pi|_{\restrict{E}{W}})$ is a twist over $\restrict{G}{W}$.

The restrictions $\iota|_{W\times\TT}$ and $\pi|_{\restrict{E}{W}}$ are still continuous groupoid homomorphisms, and  $\iota|_{W\times\TT}$ is still injective and its range is contained in $\restrict{E}{W}$. 
    Restricting a homeomorphism to a subset of its domain yields a homeomorphism onto the image of that subset; thus, $\iota|_{W\times\TT}$ is a homeomorphism onto $\iota(W\times\TT)=\pi\inv (W)=\pi\inv (\restrict{E}{W}\z )$.
It is immediate that $\iota|_{W\times\TT}(W\times \TT)$ is central in $\restrict{E}{W}$ because it  is central in $E$.
To see that $\pi|_{\restrict{E}{W}}$ is open, let $U$ be open in $\restrict{E}{W}$. Then there exists an open subset $U'$ of $E$ such that $U=U'\cap \restrict{E}{W}$. Then
\begin{align*}
    \pi|_{\restrict{E}{W}}(U)=\pi(U'\cap \restrict{E}{W})=\{\pi(e) : e\in U', r_E(e), s_E(e)\in W\}=\pi(U')\cap \restrict{G}{W}
\end{align*}
because $\pi$ is range and source preserving. Since $\pi$ is open, $\pi(U')$ is open in $G$, and hence $\pi|_{\restrict{E}{W}}(U)$ is open in $\restrict{G}{W}$. 
Thus, $\pi|_{\restrict{E}{W}}$ is open and $(\restrict{E}{W}, \iota|_{W\times\TT}, \pi|_{\restrict{E}{W}})$ is a twist over $\restrict{G}{W}$. 

Clearly,  if $G$ is principal, then so is $G|_W$.
Now suppose that $G$ is r-discrete, i.e.,  $G\z  $ is open in $G$. Hence, $W=\restrict{G}{W}\cap G\z  $ is open in $\restrict{G}{W}$ and so $\restrict{G}{W}$ is r-discrete. If $G$ is even \etale, then since the counting measures form a Haar system for $G$, they also form a Haar system for $\restrict{G}{W}$.  It follows from \cite[Proposition~1.29]{Williams:groupoid} that  $\restrict{G}{W}$ is also \etale.
\end{proof}

\subsubsection{An induced Haar system on a twist}

In this paper we consider $r$-discrete as well as \etale\ groupoids $G$. To consider a $\Cst$-algebra of $G$, we need to assume that $G$ admits a Haar system. But an r-discrete groupoid with a Haar system is \etale\ by \cite[Propositions~1.23 and 1.29]{Williams:groupoid}, and then we may as well use the  Haar system of counting measures. So whenever we consider a $\Cst$-algebra of $G$ or a twisted groupoid $\Cst$-algebra over $G$, then we will assume that $G$ is \etale\ and we will write $\Cst_\red (G)$ and $\Cst(G)$ for its reduced respectively full $\Cst$-algebra. Moreover, for a twist over $G$, we will also always use the induced Haar system  that we now construct. See also \Cref{counting measures good}.

Let $(E, \iota, \pi)$ be a twist over a \LCH\ groupoid $G$, let~$\lambda$ be normalized Haar measure on $\TT$, and assume that $G$ is \etale. Fix $x\in E\z  =G\z  $ and let $\gamma\in xG$. For each $e\in\pi\inv (\gamma)$ there is a homeomorphism $\rho_{e}\colon\TT\to\pi^{\inv} (\gamma)$ given by $z\mapsto z\cdot e$. 
We equip $\pi^{\inv} (\gamma)$ with the measure $\sigma^{x, \gamma}$ defined by $\sigma^{x, \gamma}(U)=\lambda(\rho_{e}{\inv}(U))$, where $e\in \pi\inv(\gamma)$ is arbitrary. This notation is justified, for if $\pi(d)=\pi(e)$, then $d=w\cdot e$ for some $w\in\TT$, and since $\lambda$ is rotation invariant  it follows that $\sigma^{x, \gamma}$ indeed does not depend on the choice of  $e\in\pi\inv (\gamma)$. Because $xE$ is the disjoint union $\bigsqcup_{\gamma\in xG}\pi\inv (\gamma)$ we obtain a measure $\sigma^x$ on $E$, taking values in $[0,\infty]$ and with support in $xE$, such that \begin{equation}\label{sigma}\sigma^x(U)=\sum_{\gamma\in xG}\sigma^{x, \gamma}(U).
\end{equation}

The following \namecref{Haar system} is well known; we provide a proof for completeness.

\begin{lemma}\label{Haar system} Let $(E, \iota, \pi)$ be a twist over a \LCH\ and \etale\ groupoid~$G$. Let $\mathfrak{s}\colon G\to E$ be a (not necessarily continuous) section  of $\pi$ and for $x\in G\z  $ let $\sigma^x$ be the measure in ~\eqref{sigma}. Then for $f\in C_c(E)$ we have
\begin{equation}\label{sigma again}
\int f(e)\, \dd \sigma^x(e)=\sum_{\gamma\in xG}\int_\TT f\big(z\cdot\mathfrak{s}(\gamma)\big)\, \dd\lambda(z).
\end{equation}
Moreover, $\{\sigma^x : x\in G\z  \}$ is a left Haar system on $E$.
\end{lemma}

\begin{proof}
Fix $x\in G\z$. By definition, each $\sigma^x$ has support contained in $xE$ and the formula~\eqref{sigma again} holds. 
Next, fix $e\in xE$ and let  $U$ be a precompact neighborhood of $e$.
For each $\gamma\in xG$, choose $e_\gamma\in\pi\inv (\gamma)$. Since $\rho_{e_\gamma}\inv (U)$ is open in $\TT$ and since $\lambda$ is a Haar measure, we get
 \[
 \sigma^x (U)=\sum_{\gamma\in xG}\sigma^{x,\gamma}(U)=\sum_{\gamma\in xG}\lambda (\rho_{e_\gamma}\inv(U)) >0.
\]
Thus, the support of $\sigma^x$ is $xE$. 

Next, we show that the collection of measures is left-invariant, i.e., that for $d\in E$ and $f\in C_c(E)$ we have 
\[
\int f(de)\, \dd \sigma^{s(d)}(e)=\int f(e)\, \dd \sigma^{r(d)}(e).
\]
The integrand on the left-hand side is the function $e\mapsto f(de)$. Thus, using \eqref{sigma again}
\begin{align*}
   \int f(de)\, \dd \sigma^{s(d)}(e)&= \sum_{\gamma\in s(d)G}\int_\TT f\big(d(z\cdot\mathfrak{s}(\gamma))\big)\, \dd\lambda(z)\\
   &= \sum_{\gamma\in s(d)G}\int_\TT f\big(z\cdot(d\mathfrak{s}(\gamma))\big)\, \dd\lambda(z).
   \intertext{Since $\pi\big( d\mathfrak{s}(\gamma) \big)=\pi(d)\gamma=\pi\big(\mathfrak{s}(\pi(d)\gamma)\big)$ there exists a unique $z_\gamma\in \TT$ such that $d\mathfrak{s}(\gamma)=z_\gamma\cdot \mathfrak{s}(\pi(d)\gamma)$. Together with our above computation, we get
   }
  \int f(de)\, \dd \sigma^{s(d)}(e)&= 
  \sum_{\gamma\in s(d)G}\int_\TT f\big(z\cdot \big( z_\gamma\cdot \mathfrak{s}(\pi(d)\gamma)\big)\big)\, \dd\lambda(z)\\
  &=\sum_{\gamma\in s(d)G}\int_\TT f\big( w\cdot \mathfrak{s}(\pi(d)\gamma)\big)\, \dd\lambda(w)\\
  \intertext{by the change of variable $w=zz_\gamma$. Since $\gamma\mapsto \pi(d)\gamma$ is a bijection of
$s(d)G$ onto $r(d)G$ we get}
  \int f(de)\, \dd \sigma^{s(d)}(e)&=\sum_{\alpha\in r(d)G}\int_\TT f\big( w\cdot \mathfrak{s}(\alpha)\big)\, \dd\lambda(w)
  =\int f(e)\, \dd \sigma^{r(d)}(e),
\end{align*}
as needed.

Next, we show that \[E\z   \to \CC, \quad x\mapsto \int f(e)\, \dd \sigma^x(e),\] is continuous for any fixed $f\in C_c (E)$. Let $x_n\to x$ in $E\z$; we need to show that 
\begin{equation}\label{cty Haar}
\Big|\sum_{\gamma\in x_nG}\int_\TT f(z\cdot\mathfrak{s}(\gamma))\, \dd\lambda(z)-\sum_{\alpha\in xG}\int_\TT f(z\cdot\mathfrak{s}(\alpha))\, \dd\lambda(z)\Big|\to 0
\end{equation}
as $n\to\infty$. We will show  that $\hat f\colon G\to \CC$ defined by $\hat{f} (\gamma)=\int_\TT f(z\cdot\mathfrak{s}(\gamma))\, \dd\lambda(z)$ is in $C_c(G)$; then \eqref{cty Haar} follows because the counting measures form a Haar system on $G$.

Fix $\epsilon>0$ and $\gamma\in G$, and suppose that $\gamma_n\to \gamma$ in $G$. 
Since $E$ is a twist over $G$ there exist a neighborhood $U_\gamma$ and  a continuous  $S\colon U_\gamma\to E$ such that $\pi\circ S=\id_{U_\gamma}$.  Let $\beta\in U_\gamma$. Then $\pi(\mathfrak{s}(\beta))=\beta=\pi(S(\beta))$, and hence there exists $z_\beta\in \TT$ such that $\mathfrak{s}(\beta)=z_\beta \cdot S(\beta)$. Now
\begin{equation}\label{eq:mfs replaced by S}
\hat f(\beta)=\int_\TT f(zz_\beta\cdot S(\beta))\, \dd\lambda(z) =\int_\TT f(w\cdot S(\beta))\, \dd\lambda(w)
\end{equation}
via the change of variable $zz_\beta\mapsto w$.  In other words, in the neighborhood $U_\gamma$ of $\gamma$, we were able to write $\hat{f}$ using the continuous but only locally defined section $S$ instead of  the not necessarily continuous but global section $\mathfrak{s}$.  For fixed $z\in\TT$, the function $\beta\mapsto f(z\cdot S(\beta))$ is continuous on $U_\gamma$. So there exists a neighborhood $V\subset U_\gamma$ of $\gamma$ such that
\[\eta\in V\implies  |f(z\cdot S(\eta))-f(z\cdot S(\gamma))|<\epsilon.
\]
Then for all  $n$ large enough such that $\gamma_n\in V\subset U_\gamma$, it follows with Equation~\eqref{eq:mfs replaced by S}
\begin{align*}
|\hat f(\gamma_n)-\hat f(\gamma)|&= \Big|\int_\TT\big( f(z\cdot S(\gamma_n))-f(z\cdot S(\gamma))\big)\, \dd\lambda(z)  \Big|\\
&\leq \int_\TT\big| f(z\cdot S(\gamma_n))-f(z\cdot S(\gamma))  \big|\dd\lambda(z) <\epsilon .
\end{align*}
Thus, $\hat f$ is continuous on $G$.  To see that $\hat f$ has compact support, suppose that $\hat f(\beta)\neq 0$. Then there exists $z\in\TT$ such that $z\cdot\mathfrak{s}(\beta)\in\supp f$ and so $\beta=\pi(z\cdot\mathfrak{s}(\beta))\in \pi(\supp f)$. It follows that $\hat f$ has support contained in $\pi(\supp f)$, which is compact since $f\in C_c (E)$.  Thus, $\hat f\in C_c(G)$ and \eqref{cty Haar} follows. We have shown that $\{\sigma^x : x\in G\z  \}$ is a Haar system.
\end{proof}

Given a twist $(E,\iota,\pi)$ over an \etale\ groupoid $G$, we will always choose the above Haar system on $E$ and therefore write $\Cst(E)$ and $\Cst_\red (E)$ instead of $\Cst(E,\sigma)$ and $\Cst_\red(E,\sigma)$.

\subsubsection{The $\Cst$-algebras associated to a twist}
Let $(E, \iota, \pi)$ be a twist over a \LCH\ and \etale\ groupoid $G$, let $\{\sigma^x : x\in G\z  \}$ be the left Haar system defined in \eqref{sigma} and let $\mathfrak{s}\colon G\to E$ be a  (not necessarily continuous) section  of $\pi$.
The full and reduced \emph{twisted groupoid $\Cst$-algebras}  $\Cst(E;G)$ and $\Cst_\red(E;G)$ are obtained from $\Cst(E)$ and $\Cst_\red(E)$ as follows. Set\footnote{In other places in the literature,  $C_c(E;G)$ is the subset of $C_c(E)$ whose elements satisfy $f(z\cdot e)=\bar{z}f(e)$ instead. Our convention is consistent with \cite{Clark-anHuef-RepTh,Kum:Diags, Muhly-Williams-CtsTr2}.}
\[
C_c(E;G)\coloneqq\{f\in C_c(E) : f(z\cdot e)=zf(e)\text{ for all $z\in\TT$ and $e\in E$}\}.
\]
It is straightforward to check that $C_c(E;G)$ is an ideal in $C_c(E)$: for example,  if $f\in C_c(E)$ and $g\in C_c(E;G)$, then for $e\in E$ and $z\in \TT$ we have
\begin{align*}
g*f(z\cdot e)&= \int_E g(d) f(d\inv (z\cdot e) ) \, \dd\sigma^{r(z\cdot e)}(d)\\
\intertext{which, using the change of variable $c=d\inv (z\cdot e)$ becomes}
&=\int_E g(z\cdot ec\inv ) f(c) \, \dd\sigma_{s( e)}(c)\\
&= z\int_E g(ec\inv ) f(c) \,  \dd\sigma_{s( e)}(c)\\
&=z(g*f)(e)
\end{align*} by another change of variable. Thus, $C_c(E; G)$  extends to an ideal $\Cst(E;G)$ of $\Cst(E)$.
By \cite[Lemma~3.3]{Renault:Reps}, the map $\Upsilon \colon C_c(E)\to C_c(E;G)$ defined by
\[
\Upsilon(f)(e)=\int_\TT f(z\cdot e)\bar z\, \dd z
\]
is the identity on $C_c(E;G)$, a surjective $*$-homomorphism which is continuous in the inductive limit topology, and hence extends to a  homomorphism $\Upsilon\colon \Cst(E)\to \Cst(E;G)$. Thus, $\Cst(E;G)$ is also a quotient of $\Cst(E)$, and hence is a direct summand.

Similarly, $\Cst_\red(E;G)$ is the completion of $C_c(E;G)$ in $\Cst_\red(E)$, and is a direct summand of $\Cst_\red(E)$.
Let \[L^2(Ex;Gx)=\{\xi\in L^2(Ex,\sigma_x) : \xi(z\cdot e)=z\xi(e) \text{\ for $e\in E$ and $z\in\TT$}\}.\] 
For each $x\in G\z  $, let $L^x\colon \Cst(E)\to B(L^2(Ex,\sigma_x))$  be the extension to $\Cst(E)$ of the representation defined at \eqref{Lrep}. We   write $\pi^x$ for the nondegenerate part of $L^x|_{\Cst(E;G)}$, so that 
\begin{equation}\label{repn pi}\pi^x\colon \Cst(E; G)\to B(L^2(Ex;Gx)).
\end{equation}
In particular, $\Cst_\red(E;G)$ is the completion of $C_c(E;G)$ with respect to the norm defined by \[
\|f\|_{\Cst_\red(E;G)}=\sup\{\|\pi^x(f)\| : x\in G\z \}\]
for $f\in C_c(E;G)$.

It is useful to note that the integral formula for convolution simplifies to a sum: for $f, g\in C_c(E;G)$   we have
\begin{align*}
    f*g(e)&=\int_E f(d)g(d\inv e)\,  \dd\sigma^{r(e)}(d)\\
    &=\sum_{\gamma\in r(e)G}\int_\TT f(z\cdot \mathfrak{s}(\gamma))\,g(\bar z\cdot \mathfrak{s}(\gamma)\inv e)\, \dd z\\
    &=\sum_{\gamma\in r(e)G}  f( \mathfrak{s}(\gamma))\, g(\mathfrak{s}(\gamma)\inv e)
    \end{align*}
   since $f,g$ are $\TT$-equivariant and the measure of $\TT$ is $1$. Similarly, for $\xi\in L^2(Ex;Gx)$ we have
    \[
    \|\xi\|^2=\int_E|\xi(e)|^2\, \dd\sigma_x(e)=\sum_{\gamma\in Gx}\int_\TT|\xi(z\cdot\mathfrak{s}(\gamma))|^2\, \dd(z)=\sum_{\gamma\in Gx}|\xi(\mathfrak{s}(\gamma))|^2.
\]

Finally, for $h\in C_c(G\z )$ define a function $f_h$ on $E$ by 
\begin{equation}\label{eq-diagonal}
f_h(e)=\begin{cases} z h(x) &\text{if $e = \iota(x, z)$ for some $(x,z)\in G\z \times\TT$}\\
0&\text{else.}\end{cases}
\end{equation} 
Then $f_{h}\in C_c(E;G)$ and $h\mapsto f_{h}$ extends to an isomorphism from $C_0(G\z )$ onto the closure of the $*$-subalgebra  $\{f\in C_c(E;G) : \supp f\subset\iota(G\z \times\TT )\}$ of $\Cst_\red(E;G)$.

\bigskip
We will use the following lemma.
\begin{lemma}[{\cite[Lemma 2.7]{BFRP:Twisted}}]\label{lem-right-up} 
Let $G$ be a \LCH\  and  \etale\ groupoid and  let $(E, \iota, \pi)$ be a twist over $G$. Suppose that $H$ is an open subgroupoid of~$G$. Then the open subgroupoid $\pi\inv (H)$ of $E$ gives a twist  $(\pi\inv (H), \iota|_{H\z \times\mathbb{T}}, \pi|_{\pi\inv (H)})$ over $H$ and there is an injective  homomorphism from $\Cst_\red(\pi\inv (H);H)$ into $\Cst_\red(E;G)$ induced by inclusion and extension by zero.
\end{lemma}
The above twist over $H$ is often denoted by $E_H$. But to avoid confusion with $\restrict{E}{W}=r\inv (W)\cap s\inv (W)$ for $W\subset E\z$, we use the notation $\pi\inv(H)$ instead.

\subsection{Topological, nuclear, and dynamic asymptotic dimension}

Nuclear dimension and its predecessor, decomposition rank, arose as key structural properties in the classification program for separable nuclear $\Cst$-algebras.  Both were modeled on Lebesgue covering dimension for topological spaces. 

Let $X$ be a topological space.  An open cover $\mathcal{U}$ of $X$ has \emph{order} $m$ if each element of $X$ belongs to at most $m$ elements of $\mathcal{U}$, i.e., $\bigcap_{i=1}^{m+1} U_i=\emptyset$ for any distinct $U_{1},\ldots ,U_{m+1}\in \mathcal{U}$. The \emph{Lebesgue} (or \emph{topological}) \emph{covering dimension} of $X$, written $\dim X$, is the  smallest integer~$N$ such that every open cover of $X$ admits an open refinement of order $N+1$.

We let $\cX\coloneqq X\cup \{\infty\}$ denote the \emph{Alexandrov} (or \emph{minimal one-point}) \emph{compactification} of the topological space $X$. Its topology consists of open sets in $X$ and all sets of the form $(X\setminus K)\cup\{\infty\}$, where $K$ is 
a closed and compact subset of $X$, making $\cX$ compact and $X$ an open subspace. If $X$ is locally compact and Hausdorff, then  $\cX$ is Hausdorff.
We will need the following \namecref{dimension of unitization}, which may be well known but does not seem to be well documented. 

\begin{lemma}\label{dimension of unitization}
Let $X$ be a second-countable, locally compact and Hausdorff space, and let $\cX $ be its one-point compactification. Then $\dim X=\dim \cX $.
\end{lemma}

\begin{proof}
Since $\cX $ is second-countable, compact and Hausdorff, it is normal \cite[Theorem~32.1]{Munkres} and metrizable \cite[Theorem~34.1]{Munkres}. Now $X$ is a
subspace
of the metrizable space $\cX $, and hence $\dim X\leq \dim \cX $ \cite[Theorem~1.8.3]{Coornaert}. Further, in a metrizable space every open set is an F$_\sigma$ set \cite[Lemma~1.8.2]{Coornaert}. Now $\cX =X\cup\{\infty\}$ is a union of two F$_\sigma$ sets of dimension at most~$\dim X$. Thus, $\dim \cX \leq\dim X$ by \cite[Proposition~5.3, Chapter 3]{Pears:dim}.
\end{proof}

Now let $A$ and $B$ be $\Cst$-algebras. Recall that a linear map $\phi\colon A\to B$ is {\em positive} if it maps positive elements in $A$ to positive elements in $B$. It is {\em completely positive} (c.p.) if this also holds for all matrix amplifications $\phi^{(n)}\colon M_n(A)\to M_n(B)$ where $\phi^{(n)}([a_{ij}])=[\phi(a_{ij})]$. A completely positive norm-decreasing map is called {\em completely positive contractive} (\cpc).  A c.p.\ map  $\phi\colon A\to B$ between $\Cst$-algebras is called \emph{order zero} if for any 
    positive elements $a,b\in A$
with $ab=0$, we have $\phi(a)\phi(b)=0$.

One indispensable result for \cpc\ maps is Arveson's extension theorem \cite{Arv69}. For easy reference, we state 
the particular consequence that
we use in this article. 
\begin{thm}[Arveson]\label{thm: arv for us}
Let $A$ and $D$ be unital $\Cst$-algebras, and let $B\subset A$ be a $\Cst$-subalgebra. Then any \cpc\ map $\varphi\colon B\to D$ extends to a \cpc\ map $\psi\colon  A\to D$. 
\end{thm}

In case $1_A\in B$, the reader is referred to \cite[Theorem 1.6.1]{BrOz:FinDimApp} for a proof. If $B$ is not unital, then we may replace $B$ with $\Cst(B,1_A)$, which is canonically isomorphic to the minimal unitization $\widetilde{B}$ of $B$, and replace $\varphi$ with its \cpc\ unitization $\tilde{\varphi}\colon \Cst(B,1_A)\to D$ given by $\tilde{\varphi}(x+\lambda 1_A)=\varphi(x)+ \lambda 1_D$ for $x \in B$, $\lambda \in \CC$, as in \cite[Proposition 2.2.1]{BrOz:FinDimApp}. If $1_A\notin B$ but $B$ still has a unit $1_B\in B$, then we consider the ``forced unitization'' $\Cst(B,1_A)\cong B\oplus \CC$. Since the map $B\oplus \CC \to D$ given by $b\oplus \lambda\mapsto \varphi(b)$ is clearly \cpc, it follows that there exists a \cpc\ extension $\tilde{\varphi}\colon \Cst(B,1_A)\to D$ of $\varphi$. Then we may apply \cite[Theorem 1.6.1]{BrOz:FinDimApp} to extend $\tilde{\varphi}\colon \Cst(B,1_A)\to D$ to a \cpc\ map $\psi\colon A\to D$ where $\psi|_B=\tilde{\varphi}|_B=\varphi$.

\begin{defn}[{\cite[Definition 2.1]{WinterZacharias:nuclear}}]\label{def 2.1 WZ}
Let $A$ be a $\Cst$-algebra and let $d\in \NN$. Then $A$ has \emph{nuclear dimension} at most $d$, written $\dim_{\text{nuc}}(A)\leq d$, if for any finite subset $\Ff\subset A$ and $\varepsilon>0$, there exist a finite-dimensional $\Cst$-algebra $F=\bigoplus_{j=0}^d F_j$, a \cpc\ map $\psi\colon A\to F$ and a c.p.\ map $\phi\colon F\to A$ such that $\phi_j\coloneqq \phi|_{F_j}$ is \cpc\ and  order zero for each $0\leq j\leq d$, and  for all $a\in \Ff$,
\[\|\phi(\psi(a))-a\|<\varepsilon.\]
\end{defn}

A subtle yet stark strengthening of nuclear dimension is decomposition rank, where we add the assumption that $\phi$, and not just each $\phi_j$, is contractive:
\begin{defn}[{\cite[Definition~3.1]{KirWin}}]
 A $\Cst$-algebra $A$ has \emph{decomposition rank}  at most  $d\in\NN$, written $\dr(A)\leq d$,  if for any finite subset $\Ff\subset A$ and $\varepsilon>0$, there exists a finite-dimensional $\Cst$-algebra $F=\bigoplus_{j=0}^d F_j$ and \cpc\ maps $\psi\colon A\to F$ and $\phi\colon F\to A$ such that  $\phi_j\coloneqq\phi|_{F_j}$ is order zero for each $0\leq j\leq d$ and such that for all $a\in \Ff$,
 \[\|\phi(\psi(a))-a\|<\varepsilon.\]
\end{defn}

There are numerous reformulations of nuclear dimension and decomposition rank. In \cite[Definition 8.1]{GWY:2017:DAD}, for example,  \citeauthor{GWY:2017:DAD} use a version of finite nuclear dimension which is also more fitting to our purposes.  For the reader's benefit, we provide a proof that \cite[Definition 8.1]{GWY:2017:DAD} and \Cref{def 2.1 WZ} are equivalent. 

\begin{lemma}
\label{lemma: Same defns}
Let $A$ be a $\Cst$-algebra and let $d\in \NN$. Then \mbox{$\dim_{\nuc}(A)\leq d$} if and only if for any finite subset $\Ff\subset A$ and $\varepsilon>0$, there exist finite-dimensional $\Cst$-algebras $F_0,\ldots,F_d$ and \cpc\ maps $A\xrightarrow{\psi_j}F_j \xrightarrow{\phi_j}A$ with each $\phi_j$ order zero   such that for all $a\in \Ff$,
 \[\left\|\textstyle{\sum_{j=0}^d} (\phi_j\circ \psi_j)(a)-a\right\|<\varepsilon.\]
\end{lemma}

\begin{proof}

Fix $\Ff\subset A$ finite and $\varepsilon>0$. 

First, assume that $\dim_\nuc(A)\leq d$. Then there exist a finite-dimensional $\Cst$-algebra $F$, c.p.\ maps $A\xrightarrow{\psi}F \xrightarrow{\phi}A$ such that $\psi$ is \cpc, $F$ decomposes as $\bigoplus_{j=0}^d F_j$, $\phi_j\coloneqq \phi|_{F_j}$ is \cpc\ and order zero for each $0\leq j\leq d$, and for all $a\in \Ff$,
\begin{align*}
    \| \phi(\psi(a))-a\|&<\varepsilon.
\end{align*}
For each $0\leq j\leq d$, let  $1_{F_j}$ be the projection onto the $j^\text{th}$ summand and set $\psi_j(\cdot)\coloneqq 1_{F_j} \psi (\cdot)1_{F_j}$. Then each $\psi_j$ is \cpc, and for $a\in \Ff$, 
\[\psi(a)=\sum_{j=0}^d \psi_j(a) =\bigoplus_{j=0}^d \psi_j(a), 
\]
so that 
\[\varepsilon>\|\phi(\psi(a))-a\|=\Big\|\phi\Big(\sum_{j=0}^d \psi_j(a)\Big)-a\Big\|=\big\| \sum_{j=0}^d (\phi_j\circ \psi_j)(a) -a\big\|, 
\]
as needed.

\medskip

Second, assume that there exist finite-dimensional $\Cst$-algebras $F_0,\ldots,F_d$ and \cpc\ maps $A\xrightarrow{\psi_j}F_j \xrightarrow{\phi_j}A$ with each $\phi_i$ order zero  such that for all $a\in \Ff$,
 \begin{align*}
     \left\|\textstyle{\sum_{j=0}^d} (\phi_j\circ \psi_j)(a)-a\right\|&<\varepsilon.
 \end{align*}
 Set $F\coloneqq \bigoplus_{j=0}^d F_j$ and $\psi\coloneqq \oplus_{j=0}^d \psi_j$. Extend each $\phi_j$ to $F$ by setting $\phi_j\equiv 0$ off $F_j$, and set $\phi\coloneqq \sum_{j=0}^d \phi_j$. Then for $a\in \Ff$, 
 \[\|\phi(\psi(a))-a\|=\|\phi\big(\sum_{j=0}^d \psi_j(a)\big)-a\|=\big\|\sum_{j=0}^d \phi_j(\psi_j(a))-a\big\|<\varepsilon.\]
 Thus, $\dim_\nuc(A)\leq d$, as needed.
\end{proof}

\begin{defn}[{\cite[Definition 5.1]{GWY:2017:DAD}}]\label{defn-dad}
Let~$G$ be a \LCH\ and \etale\ groupoid. Then~$G$ has \emph{dynamic asymptotic dimension} $d\in\NN$ if $d$ is the smallest natural number with the property that for every open and precompact subset $V\subset G$, there are open subsets $U_0, U_{1},\dots, U_d$ of $G\z $ that cover $s(V)\cup r(V)$ such that for each $i\in\{0,\dots, d\}$ the set $\{\gamma \in V : s(\gamma), r(\gamma)\in U_i\}$ is contained in a precompact subgroupoid of~$G$. If no such $d$ exists, then we say that $G$ has infinite dynamic asymptotic dimension. We write $\dad(G)$ for the dynamic  asymptotic dimension of $G$.
\end{defn}

In Definition~\ref{defn-dad} we could have equivalently asked that for each $0\leq i\leq d$ the subgroupoid \( \langle V\cap \restrict{G}{U_i}\rangle\) generated by the set $\{\gamma\in V  : s(\gamma),r(\gamma)\in U_i\}$ is precompact. Notice that each $\langle V\cap \restrict{G}{U_i}\rangle$  is open in~$G$.
If the unit space of~$G$ is compact, then we can relax the definition of dynamic asymptotic dimension by only considering precompact sets $V$ that contain $G\z $.

\begin{remark}\label{compute-dad}
To compute the dynamic asymptotic dimension of a \LCH\ and \etale\ groupoid $G$, many helpful results have been established. By \cite[Example~5.3]{GWY:2017:DAD}, $\dad(G)=0$  if and only if $G$ is locally finite in the sense that it is  an increasing union of open,
precompact
subgroupoids. (In particular, the dynamic asymptotic dimension of a compact, Hausdorff and \etale\ groupoid is~$0$.) 

In \cite{GWY:2017:DAD}, Guentner, Willett and Yu also defined the notion of dynamic asymptotic dimension for an action of a discrete group 
$\Gamma$
on a \LCH\ space $X$, and they showed that the associated
transformation-group groupoid $X\rtimes\Gamma$ has dynamic asymptotic dimension equal to the dynamic asymptotic dimension of the group action \cite[Lemma~5.4]{GWY:2017:DAD}. In particular, if
 $\mathbb{Z}$
acts minimally on an infinite compact space $X$, then
$\dad(X\rtimes \mathbb{Z})=1$ 
by \cite[Theorem~3.1]{GWY:2017:DAD}. There are more results about dynamic asymptotic dimension of transformation-group groupoids in, for example,    \cite{Amini-et-al, Pilgrim-dad, Pilgrim-dad-geometry, Pilgrim-dad-sharp}; see also
\Cref{prop finite dad for ANY transformation groupoid} below for results about a transformation groupoid where a groupoid acts on a space.

When $X$ is a coarse space, the associated coarse
groupoid has dynamic asymptotic dimension equal to the asymptotic dimension of $X$ by \cite[Theorem~6.4]{GWY:2017:DAD} and the dynamic asymptotic dimension of a graph groupoid is either $0$ or $\infty$ by \Cref{lemma-directed-graph-groupoids} below.

Recently, some very useful reduction results have appeared in \cite[Theorem~A]{Bo:2023:DAD-pp}:  equivalent \etale\ Hausdorff groupoids  have the same  dynamic asymptotic dimension. Moreover, if $U_1, \dots, U_n$ is an open cover of the unit space of 
such a groupoid
$G$, then $\dad(G)$ is the maximum of the $\dad(G|_{U_i})$.
Finally, by \Cref{thm: unitization summary} below, for any \LCH\ \etale\ groupoid $G$ there exists a \LCH\ \etale\ groupoid $\cG$ with compact unit space such that $\dad(G)=\dad(\cG)$, and that is a key tool in the proof of our main theorem, \Cref{holy grail}. \end{remark}

\section{Unitizations}\label{unitization}
    For every $\Cst$-algebra $A$ there exists a unique unital $\Cst$-algebra $\widetilde A$, called the \emph{minimal unitization} of $A$,  that contains $A$ as an ideal with $\widetilde A/A$ isomorphic to $\CC$ via  $1+A\mapsto 1$. A favorite technique among $\Cst$-algebraists is to reduce a proof to the unital setting using the minimal unitization. In this section, we describe how such a unitization for a (twisted) groupoid $\Cst$-algebra is realized at the level of the groupoid by taking the Alexandrov one-point-compactification of the unit space. Our construction in the non-twisted case agrees with the Alexandrov groupoid of \cite{KKLRU}. Just as for the associated $\Cst$-algebras, the original groupoid will share many key properties with its associated Alexandrov groupoid, such as dynamic asymptotic dimension and topological  covering dimension. The following \namecref{thm: unitization summary} summarizes our main results for \etale\ groupoids, though some of our results are shown for r-discrete groupoids (see \Cref{counting measures good} on our hypotheses). 
    
 \begin{thm}\label{thm: unitization summary}
 Let $(E, \iota, \pi)$ be a twist over a \LCH\ and \etale\ groupoid $G$ with non-compact unit space $G\z $. Then there exists a \LCH\ and \etale\ groupoid~$\cG$ with compact unit space and a twist $(\widetilde E,\widetilde\iota, \widetilde{\pi})$ over $\cG$ such that the minimal unitization of $\Cst_\red(E;G)$ is isomorphic to $\Cst_\red(\widetilde E; \cG)$. Similarly, the minimal unitization of $\Cst(E;G)$ is isomorphic to $\Cst(\widetilde E; \cG)$. Moreover, $\dad(G)=\dad(\cG)$ and $\dim(G\z )=\dim(\cG\z )$.
 \end{thm}
 
 \begin{remark}\label{counting measures good}
 We state and prove \Cref{thm: unitization summary} for $\Cst$-algebras of \etale\ groupoids instead of r-discrete groupoids for the following reason. To consider a groupoid $\Cst$-algebra of $G$, we need to assume that $G$ carries a Haar system, and to consider $\cG$, we need to assume that  $G$ is r-discrete. These two assumptions combined imply that $G$ is \etale\ by \cite[Propositions~1.23 and~1.29]{Williams:groupoid}.
 \end{remark}

We  call the groupoid $\cG$ the \emph{Alexandrov groupoid} of~$G$ after \cite[Definition 7.7]{KKLRU}, and we call the twist $(\widetilde E, \widetilde\iota, \widetilde{\pi})$ over $\cG$ the \emph{Alexandrov twist}. We will prove \Cref{thm: unitization summary} via a series of lemmas below. In \Cref{sect: Alexandrov}, we give precise descriptions of the Alexandrov groupoid and the Alexandrov twist and establish their relevant groupoid properties (\Cref{lem:etale:tilde G:topologically'} and \Cref{tilde E' v2}). In \Cref{sect: cstar unitization}, we show that the correspondence between the unitizations of the groupoid $\Cst$-algebras and the Alexandrov groupoid $\Cst$-algebras. Finally, in \Cref{sect: dad for Alexandrov}, \Cref{DAD of Gtilde'} proves our claim about dynamic asymptotic dimension.

\subsection{The Alexandrov groupoid and twist}
\label{sect: Alexandrov}

In order to elucidate some topological minutiae that will feature in subsequent proofs, we state the following lemma for easy reference. 

\begin{lemma}\label{lem:cpctness}
Let $X$ be a topological space and let $Y$ be a subspace. Then for $C\subset Y$, we have the following. 
\begin{enumerate}[label=\textup{(\arabic*)}]
    \item\label{cpctness:item:cpct in X vs Y} $C$ is compact in $X$ if and only if it is compact in $Y$. 
    \item\label{cpctness:item:precpct in X vs Y} If $\cl[X]{C}\subset Y$, then $C$ is precompact in $X$ if and only if it is precompact in $Y$.
    \item\label{cpctness:item:precpct in Y} If $X$ is Hausdorff and $C$ is precompact in $Y$, then $\cl[X]{C}=\cl[Y]{C}\subset Y$ and $C$ is precompact in $X$. 
\end{enumerate} 
\end{lemma}
\begin{proof}
A subset of a topological space is compact if and only if it is compact with respect to the subspace topology.  Since the topologies on $C$ that are induced from $X$ resp.\ from $Y$ coincide, Part~\ref{cpctness:item:cpct in X vs Y} follows.
    
For Part~\ref{cpctness:item:precpct in X vs Y}, note that if $\cl[X]{C}\subset Y$, then  $\cl[Y]{C}=\cl[X]{C}$, and it then follows from Part~\ref{cpctness:item:cpct in X vs Y} that $C$ is precompact in $X$ if and only if it is precompact in $Y$.

For Part~\ref{cpctness:item:precpct in Y}, we note that if $C$ is precompact in $Y$, then $\cl[Y]{C}$ is compact in $X$ by Part~\ref{cpctness:item:cpct in X vs Y}. Since $X$ is Hausdorff it follows that $\cl[Y]{C}$ is closed in $X$, and hence $\cl[Y]{C}=\cl[X]{C}$. Then Part~\ref{cpctness:item:precpct in X vs Y} finishes the claim.
\end{proof}

\begin{lemma}\label{lem:etale:tilde G:topologically'}
Let $G$ be a \LCH\ and r-discrete groupoid with non-compact unit space $G\z $, and let $G\z_+\coloneqq G\z \cup \{\infty\}$ denote the Alexandrov one-point compactification of $G\z $. 
Set $\cG\coloneqq G\cup \{\infty\}$, $\cG\comp \coloneqq G\comp \cup\{(\infty, \infty)\}$, and $\cG\z \coloneqq G\z_+$. We define composition of tuples in $G\comp  \subset 
    \cG\comp $ as in~$G$, and extend $r$ and $s$ to $\cG$ by setting $r\inv (\infty)=s\inv (\infty)=\{\infty\}$. The set $\mathfrak{U}$ consisting of $\cG$, all open sets in $G\z_+$ and all open sets in $G$  is a basis for a topology on $\cG$. Equipped with this basis, $\cG$ is a \LCH\ and r-discrete groupoid with compact unit space $\cG\z$, and it contains $G$ as a subgroupoid. If in addition $G$ is \etale, then so is $\cG$.
\end{lemma}

\begin{proof}
   To see that $\mathfrak{U}$ is a basis for a topology, the only non-trivial claim is that if $\alpha\in U_{1}\cap U_{2}$ for 
   $U_{1}$ an open subset of $G$ and $U_{2}$ an open subset of $\cG\z=G\z \cup \{\infty\}$, then there exists $U$ in $\mathfrak{U}$ with $\alpha \in U \subset  
    U_{1}\cap U_{2}$. Since compact sets in the Hausdorff space $G\z$ are closed, we always have that $U_{2}\setminus \{\infty\} = G\z \setminus A$ for some closed subset $A$ of $G\z$, independent of whether $\infty$ is or is not an element of $U_{2}$. Since $G$ is r-discrete, the unit space $G\z $ is open in~$G$, and so $U_{1}\cap U_{2}= U_{1}\cap (G\z\setminus A)$ is automatically open in $G$ and thus the required element $U\in \mathfrak{U}$.
   
    By definition of $\mathfrak{U}$, both~$G$ and $\cG\z =G\z \cup\{\infty\}$ are open in $\cG $. In particular, $\cG $ is r-discrete. Using nets in $\cG$ that either converge to a point in $G$ or to $\infty$, it is easy to verify that the operations and the range and source maps on $\cG $ are continuous.

  Since~$G$ is Hausdorff, to check  that $\cG$ is Hausdorff, it suffices to separate $\infty$ from an arbitrary $\beta\in G$. Since~$G$ is locally compact, we can find an open set $V$ around $\beta$ in~$G$ which is precompact in~$G$. Since $G\z $ is closed in~$G$, the closure $K$ of $V\cap G\z $ in $G\z $ is a compact subset of $G\z $. In particular, $(G\z \setminus K)\cup \{\infty\}$ is an open neighborhood of $\infty$ in $\cG $ which is disjoint from the open neighborhood $V$ of $\beta$.

Next we check that  $\cG $ is locally compact. Since $\cG\z $ is a compact open neighborhood of $\infty$ in $\cG$, we only need to check that any $\alpha\in \cG\setminus\{\infty\}$ has a precompact open neighborhood in $\cG$. Let $U$ be a neighborhood of $\alpha$ that is precompact and open in $G$. Then $U\in \mathfrak{U}$, and since $G\subset \cG$ is equipped with the subspace topology, \Cref{lem:cpctness} says that $U$ is precompact in $\cG$. Thus, $\cG$ is \LCH\ and r-discrete.

Finally, suppose that $G$ is \etale. Since $G$ is r-discrete, this is equivalent to $G$ having a basis of open bisections. Since $\mathfrak{U}$ contains all open sets of~$G$, all open bisections of $G$ are open bisections of $\cG $; since $\mathfrak{U}$ contains all open sets of $G_+\z $  and since the restriction of the range map to these sets is the identity, any open subset of $G_+\z $ containing $\infty$ is an open bisection in $\cG $. Thus, by adding all open subsets of $G_+\z $ containing $\infty$ to a basis of open bisections of $G$ we obtain a basis of open bisections of $\cG $. This concludes our proof.
\end{proof}

If $G$ is \etale,  then $\cG$ agrees with the Alexandrov groupoid $G^+$ from \cite[Definition 7.7]{KKLRU}. 
We reserve the notation $\cX$ for the Alexandrov compactification of a topological space $X$ and use 
$\cG$ for the Alexandrov groupoid of a groupoid~$G$.

\begin{example}\label{ex:transformation gpd unitised}
Let $\Gamma$ be a discrete group with identity $e$ which acts continuously on a locally compact, non-compact, Hausdorff space $X$ and let $G=\Gamma\ltimes X$ be the transformation-group groupoid. Let $\cG $ be the Alexandrov groupoid as above, so that $\cG =G\cup\{\infty\}$. Denote by $\cX$ the one-point compactification of $X$. It is easy to check that the action of $\Gamma$ 
on $\cX$ given by 
\[\gamma\cdot x=\begin{cases}\gamma\cdot x &\text{if $x\in X$}\\
\infty &\text{if $x=\infty$}\end{cases}\]
is continuous. As sets, we have 
  \[  \cG =\{(\gamma,x)\colon\gamma\in \Gamma, x\in X\}\cup\{(e,\infty)\}\text{\ and\ }
    \Gamma\ltimes \cX=\{(\gamma,x) : \gamma\in\Gamma, x\in\cX\}.
\]
Let $\iota\colon \cG\to \Gamma\ltimes \cX$ be the inclusion; it is trivial to check this is a groupoid homomorphism. Now notice that
\[
(\Gamma\ltimes\cX)\setminus \iota(\cG )=\{(\gamma, \infty) : \gamma\neq e\}
\]
is closed: for if $(\gamma_n, \infty)\to(\alpha, y)$ in $\Gamma\ltimes \cX$ with $\gamma_n\neq e$, then $\gamma_n\to \alpha$ and $\infty\to y$; thus $y=\infty$ and $\alpha\neq e$
since $\Gamma$ is discrete, and hence $\cG $ is a wide and  open subgroupoid of $\Gamma\ltimes\cX$. It follows that the ``inclusion''  $i\colon C_c(\cG )\to C_c(\Gamma\ltimes\cX)$ given  by extending functions to be $0$ off of $\cG$ is a well-defined homomorphism which is isometric for the reduced norms. Thus, $i$ extends to an injective homomorphism $i\colon \Cst_\red(\cG )\to \Cst_\red(\Gamma\ltimes\cX)$.
\end{example}

Next, we discuss the Alexandrov twist.

\begin{lemma}
\label{tilde E' v2}
 Let $(E, \iota, \pi)$ be a twist over a \LCH, r-dis\-crete  grou\-poid~$G$ with non-compact unit space $G\z $. Set $\widetilde E\coloneqq E\cup\{\infty_z : z\in\TT\}$ and $\widetilde E\comp \coloneqq E\comp \cup\{(\infty_w, \infty_z) : z,w\in\TT\}$ with composition of pairs in $E\comp \subset \widetilde E\comp $ defined as in $E$ and $\infty_w\infty_z=\infty_{wz}$ for all $w,z\in\TT$.
 Define $\widetilde r, \widetilde{s}\colon \widetilde E\to G\z \cup\{\infty\}$ by $\widetilde r|_E=r$, $\widetilde{s}|_E=s$ and $\widetilde r(\infty_w)=\infty=\widetilde{s}(\infty_w)$ for all $w\in\TT$, and define $\widetilde\iota\colon \cG\z \times\TT\to \widetilde E$ by $\widetilde\iota|_{G\z \times\TT}=\iota$ and $\widetilde\iota(\infty, z)=\infty_z$. Then we have the following. 
 \begin{enumerate}[label=\textup{(\arabic*)}]
     \item\label{tilde E 2}
     The collection $\mathfrak{V}$ consisting of $\widetilde E$, the open sets of $E$ and the sets $\widetilde\iota(V)$ where $V$ is an open subset of $\cG\z \times\TT$, is a basis for a topology on $\widetilde E$ with respect to which $\widetilde\iota$ is open and continuous.
     \item\label{tilde E 3} The topology on $E$ coincides with the subspace topology and $\widetilde E$ is a locally compact and Hausdorff groupoid.
     \item\label{tilde E 4} Define $\widetilde\pi\colon \widetilde E\to\cG$ by $\widetilde\pi|_E=\pi$ and $\widetilde\pi(\infty_z)=\infty$.  Then $(\widetilde E, \widetilde\iota, \widetilde\pi)$ is a twist over $\cG$.
 \end{enumerate}
 \end{lemma}

 \begin{proof}
For \ref{tilde E 2}, let $U_{1}, U_{2}\in \mathfrak{V}$; we will show that \[U\coloneqq U_{1}\cap U_{2}\in \mathfrak{V}.\] 
If either $U_i$ is $\widetilde E$, or if both  $U_i$ are subsets of $E$, then the claim is trivial. 
If both  $U_i$ are of the form $\widetilde\iota(V_i)$, then since $\widetilde\iota$ is injective we  again have $U=\widetilde\iota(V_{1}\cap V_{2})\in \mathfrak{V}$.

So suppose that $U_{1}$ is open in $E$ and that $U_{2}=\widetilde\iota (V)$ where $V$ is open in $\cG\z      \times\TT$. Without loss of generality, $V=W\times S$ where $W$ is open in $\cG\z$ and $S$ is open in $\TT$.  Then since $U_{1}\subset E$, we have $U=U_{1}\cap \iota(W\setminus\{\infty\}\times S)$.  As mentioned earlier, since $G$ is r-discrete, $\pi\inv (G\z)$ is open in $E$, and since $\iota$ is a homeomorphism onto $\pi\inv (G\z)$, it follows that $\iota$ is an open map into $E$. Since $W\setminus\{\infty\}\times S$ is open in $G\z \times \TT$ we get that $\iota(W\setminus\{\infty\}\times S)$ is open in $E$. So $U$ is open in $E$ and thus in $\widetilde E$ as well. Hence, $\mathfrak{V}$ is a basis for a topology $\widetilde  E$. Note that $\widetilde{\iota}$ is an open, injective map by construction.

\smallskip

To see that $\widetilde \iota$ is continuous, let $U\subset \widetilde E$ be an open set. If $U$ is open in $E$, then $\widetilde \iota\inv (U)= \iota\inv (U)$ is open in $G\z  \times \TT$ and hence in $ \cG\z \times \TT$. If $U$ is of the form $\widetilde \iota (V)$ for some open set $V$ in $ \cG\z \times\TT$, then $\widetilde\iota\inv (U) = V$ is open. Thus, $\widetilde\iota$ is continuous.

\medskip

For~\ref{tilde E 3}, to see that the topology $\tau$ on  $E$ is exactly the subspace topology $\mathfrak{V}\cap E$, note first that $\tau\subset \mathfrak{V}$. For the other inclusion, it suffices to check that, whenever $W$ and $S$ are open subsets of $\cG\z $ and $\TT$, respectively, then $\widetilde{\iota}(W\times S)\cap E$ is open in $E$.  Note that $\widetilde{\iota}(W\times S)\cap E = \iota([W\setminus \{\infty\}]\times S)$. Since $W\setminus\{\infty\}$ is open in $G\z$ and $\iota$ is open as  a map into $E$, it follows that $\widetilde{\iota}(W\times S)\cap E $ is  open in $E$.

\smallskip
It is easy to verify that $\widetilde E$ is an algebraic groupoid with unit space  $\cG\z =G\z \cup\{\infty\} $; it remains to show that $\tilde{s}, \tilde{r},$ inversion and multiplication are continuous. To that end, let $\net{e_{\lambda}}{\lambda}$ be a net in $\widetilde{E}$ that converges to $e\in \widetilde{E}$. Assume first that $e\in E$. Then since $E$ is open in $\widetilde{E}$, we eventually have $e_\lambda\in E$ and hence \[
\widetilde{s}(e_\lambda)=s(e_\lambda)\to s(e)=\widetilde{s}(e),\quad \widetilde{r}(e_\lambda)=r(e_\lambda)\to r(e)=\widetilde{r}(e) \quad\text{ and }\quad e_{\lambda}\inv \to e\inv
\]
by continuity of the source, range, and inversion map of $E$. Second, assume that $e=\infty_{z}$ for some $z\in\TT$. Then $e=\widetilde{\iota}(\infty,z)$ is in the image of $\widetilde{\iota}$. If we take a neighborhood around $e$, say $\widetilde{\iota}(V)$ for $V=[(G\z \setminus K)\cup\{\infty\}]\times S$ where $K$ is a compact subset of $G\z $ and $S$ is an open neighborhood around $z$ in $\TT$, then eventually $e_\lambda\in \widetilde{\iota}(V)$. Since $\widetilde{\iota}$ is, when restricted to~$V$, still an open map that is surjective onto its range, Fell's criterion (\cite[Proposition 1.1]{Williams:groupoid}) implies that there exists a subnet $\net{e_{f(\mu)}}{\mu}$ of $\net{e_{\lambda}}{\lambda}$ such that for each $\mu$, $e_{f(\mu)}= \widetilde{\iota} (x_\mu,z_\mu)$ for some $x_\mu\in  \cG\z $ and some $z_\mu\in \TT$ where  $x_\mu\to \infty$ and $z_\mu\to z$. Since 
\[
\widetilde{r}(e_{f(\mu)})=\widetilde{s}(e_{f(\mu)}) = x_{\mu}\to \infty = \widetilde{s}(e) =  \widetilde{r}(e) \text{ in } \cG\z  
\]
and
\[e_{f(\mu)}\inv = \widetilde{\iota} (x_\mu,\overline{z_\mu})\to \widetilde{\iota} (\infty,\overline{z}) = e\inv
\]
by continuity of $\widetilde{\iota}$, we have shown that $\net{\widetilde{s}(e_{\lambda})}{\lambda}=\net{\widetilde{r}(e_{\lambda})}{\lambda}$ has a subnet that converges to $\widetilde{s}(e)=\widetilde{r}(e)$ and that $\net{e_{\lambda}\inv}{\lambda}$ has a subnet that converges to $e\inv$. This shows that $\widetilde{s}, \widetilde{r}$, and the inversion map are continuous by, for example, \cite[Theorem~18.1]{Munkres}.

\smallskip

Suppose next that we have $(d_\lambda,e_\lambda)\to (d,e)$, a convergent net in $\widetilde{E}\comp $. Since $(d,e)$ is composable, we either have $(d,e)\in E\comp $ or $(d,e)=(\infty_{w},\infty_{z}).$ In the first case, note that since $E$ is open in $\widetilde{E}$, $E\comp $ is open in $\widetilde{E}\comp $. Thus, we eventually have $(d_\lambda,e_\lambda)\in E\comp $ and continuity of the multiplication on $E$ implies that $d_\lambda e_\lambda\to de$ in $E$ and hence in $\widetilde{E}$. In the case where $(d,e)=(\infty_{w},\infty_{z})$, consider a basic open neighborhood of $(d,e)$, say $[\widetilde{\iota}(V_{1})\times\widetilde{\iota}(V_{2})]\cap \widetilde{E}\comp $. Eventually, $(d_\lambda,e_\lambda)$ must be in that set, so we must have $d_\lambda=\widetilde{\iota}(x_\lambda,w_\lambda)$ and $e_\lambda=\widetilde{\iota}(x_\lambda,z_\lambda)$ with $x_\lambda\to \infty$ in $G\z \cup\{\infty\}$, and with $w_\lambda\to w$ and $z_\lambda\to z$ in $\TT$. By continuity of $\widetilde{\iota}$, we conclude
\[
    d_\lambda e_\lambda= \widetilde{\iota}(x_\lambda,w_\lambda) \widetilde{\iota}(x_\lambda,z_\lambda)
    =
    \widetilde{\iota}(x_\lambda,w_\lambda z_\lambda)\to\widetilde{\iota} (\infty, wz)
    =
    \infty_{wz}=\infty_w\infty_z = de.
\]

\smallskip
Because $E$ and $\TT$ are Hausdorff and the map $\iota$ is injective, to check that $\widetilde{E}$ is Hausdorff, we need only to check that we can separate a point $e\in E$ from a point $\infty_z\in \widetilde{E}\setminus E$. Since $s(e)\neq \infty$, there exists an open neighborhood $V$ of $s(e)$ in $G\z $ with compact closure $K$. Then $X\coloneqq(G\z \setminus K)\cup\{\infty\}$ is an open neighborhood of $\infty$ in $\cG\z$ with $V\cap X=\emptyset$. Then $s\inv(V)=\widetilde{s}\inv(V)$ is an open neighborhood of $e$ in $\widetilde E$ and $\widetilde\iota(X\times\TT)$ is an open neighborhood of $\infty_z$. Since $\widetilde{s}(\widetilde{s}\inv (V)\cap \widetilde\iota(X\times\TT))  \subset V\cap X=\emptyset$   we must have $\widetilde{s}\inv (V)\cap \widetilde\iota(X\times\TT)=\emptyset$ as well.

\smallskip 

Since $E\subset \widetilde{E}$ is a subspace, \Cref{lem:cpctness}~\ref{cpctness:item:cpct in X vs Y} tells us that a neighborhood of a point $e\in E$ that is open and precompact in $E$ is open and precompact in $\widetilde{E}$ as well. For $z\in \TT$, $\tilde{\iota}(\widetilde{G}\z \times \TT)$ is an open compact neighborhood of $\infty_z\in \widetilde{E}$. Thus, $\widetilde{E}$ is locally compact and Hausdorff.

\medskip

For \ref{tilde E 4}, we note first that, clearly, $\widetilde\iota$ is a homomorphism. By construction,  $\widetilde\iota$ is an injective map which is open into $\widetilde{E}$, and we have shown above that $\widetilde\iota$ is continuous. It follows that $\widetilde\iota$ is a homeomorphism onto its range  
\begin{align*}
\widetilde{\iota}(\cG\z \times\TT) =\iota(G\z \times\TT)\cup\{\infty_z : z\in\TT\}=
\pi\inv(G\z )\cup\{\infty_z : z\in\TT\}=\widetilde{\pi}\inv (\cG\z).
\end{align*}

Next, $\widetilde\pi$ is clearly a surjective homomorphism. To see that $\widetilde\pi$ is continuous, observe that for any open set $U\subset G$ the set $\widetilde\pi\inv (U)=\pi\inv(U)$ is open in $E$ and hence in $\widetilde{E}$. For a basic open neighborhood $(G\z \setminus K)\cup\{\infty\}$ around $\infty$ with $K$ compact in $G\z $, we compute
\begin{align*}
    \widetilde\pi\inv ((G\z \setminus K)\cup\{\infty\})
    &=
    \pi\inv (G\z \setminus K)\cup \{\infty_z : z\in\TT\}
    \\
    &=
    \iota ([G\z \setminus K]\times\TT) \cup \{\infty_z : z\in\TT\}
    \\
    &=
    \widetilde\iota \bigl([(G\z \setminus K)\cup\{\infty\}]\times\TT\bigr);
\end{align*}
since $[(G\z \setminus K)\cup\{\infty\}]\times\TT$ is open in $\cG\z  \times\TT$, this is an open set in $\widetilde{E}$, and $\widetilde\pi$ is continuous.

\smallskip
To see that $\widetilde\pi$ is open, let $U$ be a basic open set in $\widetilde E$. If $U=\widetilde E$, then $\widetilde\pi(U)=\widetilde{G}$ is open. If $U$ is open in $E$, then $\widetilde\pi(U)=\pi(U)$ is open in $E$ and hence in $\widetilde{E}$. Otherwise, $U=\widetilde\iota(V)$ where $V$ is open in $\widetilde{G}\z \times\TT$. We may assume that $V=W\times S$ where $W$ is open in $\widetilde{G}\z $ and $S$ is open in $\TT$. Then 
\[
\widetilde\pi (U)=\widetilde\pi (\widetilde\iota(W\times S))=W
\]
because $\pi\circ \iota(x,z)=x$ and $\widetilde\pi\circ\widetilde\iota (\infty, z)=\infty$.  Thus, $\widetilde\pi$ is open.

To see that $\widetilde{\iota}(\cG\z \times\TT)$ is central, first take $e\in E$. We have $\widetilde{\iota}(\widetilde{r}(e), z)=\iota(r(e), z)$ and so the corresponding property of $\iota$ implies $\widetilde{\iota}(\widetilde{r}(e), z)e=e\widetilde{\iota}(\widetilde{s}(e), z)$, as needed. If $e=\infty_{w}$, then the claim follows from commutativity of $\TT$: 
      \[  \widetilde{\iota}(\widetilde{r}(e), z)e
        =
        \widetilde{\iota}(\infty, z)e
        =
        \infty_z \infty_{w}
        =
        \infty_{w}\infty_z 
        =
        e\widetilde{\iota}(\widetilde{s}(e), z).\]
Thus, $(\widetilde E, \widetilde\iota, \widetilde\pi)$ is a twist over $\cG$.
\end{proof}

\subsection{The twisted \texorpdfstring{$\Cst$}{C*}-algebra of the Alexandrov groupoid}\label{sect: cstar unitization}

Next we will show that the minimal unitization of a non-unital twisted groupoid $\Cst$-algebra coincides with the twisted groupoid $\Cst$-algebra of the Alexandrov twist. We need the following lemma. 

\begin{lemma}
\label{reduced and full ses}
Let $(E, \iota, \pi)$ be a twist over a \LCH\ and \etale\  groupoid~$G$. Let  $U$ be an open and invariant subset of $G\z $ and set $W=G\z \setminus U$. Then the inclusion map on $C_c(\restrict{E}{U})$ and the restriction map on $C_c(E; G)$  induce the short exact sequence
\begin{equation}\label{ses-full}
0\longrightarrow\Cst( \restrict{E}{U};  \restrict{G}{U})\stackrel{i}{\longrightarrow}  \Cst(E;G)\stackrel{j}{\longrightarrow}  \Cst(\restrict{E}{W};\restrict{G}{W})\longrightarrow 0.
\end{equation}
If the reduced norm on $C_c(\restrict{E}{W})$ equals the universal norm, then the inclusion and restriction  maps induce  the short exact sequence
\begin{equation}\label{ses-red}
0\longrightarrow\Cst_\red( \restrict{E}{U};  \restrict{G}{U})\stackrel{i_{\red}}{\longrightarrow}  \Cst_\red(E;G)\stackrel{j_{\red}}{\longrightarrow}  \Cst_\red(\restrict{E}{W};\restrict{G}{W})\longrightarrow 0.
\end{equation}
\end{lemma}

\begin{proof} By \cite[Theorem~5.1]{Williams:groupoid}, the inclusion and restriction maps induce the short exact sequence 
\begin{equation*}
0\longrightarrow\Cst( \restrict{E}{U})\stackrel{i}{\longrightarrow}  \Cst(E)\stackrel{j}{\longrightarrow}  \Cst(\restrict{E}{W})\longrightarrow 0,
\end{equation*}
and if the reduced norm on $C_c(\restrict{E}{W})$ equals the universal norm, by \cite[Proposition~5.2]{Williams:groupoid} they also induce 
\begin{equation*}
0\longrightarrow\Cst_\red( \restrict{E}{U})\stackrel{i_{\red}}{\longrightarrow}  \Cst_\red(E)\stackrel{j_{\red}}{\longrightarrow}  \Cst_\red(\restrict{E}{W})\longrightarrow 0.
\end{equation*}
Since $U$  is invariant, for $z\in\TT$ we have $z\cdot e\in E|_U$ if and only if $e\in E|_U$.
    
It follows that the inclusion map takes  $C_c( \restrict{E}{U};  \restrict{G}{U})$ to $C_c(E;G)$. Similarly, $W$ is invariant, and the restriction map takes $C_c(E;G)$ to $C_c( \restrict{E}{W};  \restrict{G}{W})$. 
Since $\Cst( \restrict{E}{U};  \restrict{G}{U})$,  $\Cst(E;G)$ and $\Cst(\restrict{E}{W};\restrict{G}{W})$ are direct summands of $\Cst( \restrict{E}{U})$, $\Cst( E)$ and $\Cst(\restrict{E}{W})$, respectively, it follows that \Cref{ses-full} is exact.  Similarly,  \Cref{ses-red} is exact.
\end{proof}

\begin{lemma}\label{twisted unitizations'}    Let $(E, \iota, \pi)$ be a twist over a
\LCH\ and \etale\ groupoid~$G$. Then $\Cst(\widetilde E;\cG)\cong \Cst(E;G)^\sim$ and $\Cst_\red(\widetilde E;\cG)\cong \Cst_\red(E;G)^\sim$.
\end{lemma}

\begin{proof}  We will apply \Cref{reduced and full ses} to the twist $(\widetilde E, \widetilde \iota, \widetilde \pi)$ with the open and invariant subset $U=G\z $ and its complement $W=\{\infty\}$. Since the restricted groupoid $ \restrict{\widetilde E}{W}=\{\infty_w : w\in \TT\}$ is isomorphic to $\TT$, the reduced and universal norms on $C_c(\restrict{\widetilde{E}}{W})$ coincide, and so \Cref{reduced and full ses} gives short exact sequences on the level of both the full and reduced $\Cst$-algebras.

We start by proving that $C_c( \restrict{\widetilde E}{W};\restrict{\cG}{W})$ is isomorphic to $\CC$. Here  $C_c( \restrict{\widetilde E}{W}; \restrict{\cG}{W})$ consists of functions $f\colon \{ \infty_w : w\in\TT\}\to\CC$ such that  $z f(\infty_w)=f(z\cdot \infty_w)$. Define $\rho\colon C_c( \restrict{\widetilde E}{W}; \restrict{\cG}{W})\to\CC$ by $\rho(f)= f(\infty_{1})$. Let  $f,g\in C_c( \restrict{\widetilde E}{W}; \restrict{\cG}{W})$. Then for any section $\mathfrak{s}\colon \cG \to\widetilde{E}$ of $\widetilde{\pi}$ we have
 \begin{align}
 \begin{split}\label{eq:rho multiplicative}
    \rho(f\ast g)&=(f\ast g)(\infty_{1})
    =
    \sum_{\gamma\in \infty\cG } f(\mathfrak{s}(\gamma))g(\mathfrak{s}(\gamma)\inv \infty_{1})
    \\&=
    f(\infty_{z})g(\infty_{\overline{z}}\infty_{1})
    =
    z\overline{z}f(\infty_{1})g(\infty_{1})
    =\rho(f)\rho(g)
\end{split}
 \end{align}
 and $\rho(f^*)=\overline{f(\infty_{1}\inv )}=\overline{f(\infty_{1})}$. Thus, $\rho$ is a $*$-homomorphism.     Since $\cG$ is \etale, so is $\restrict{\cG}{W}$ by \Cref{restriction of twist non-etale}.  
 Thus $\rho$  extends to a homomorphism $\rho\colon \Cst(\restrict{\widetilde{E}}{W}; \restrict{\cG}{W})\to \CC$ by \cite[Lemma~2.14]{Armstrong:twistedgpds}. It is clear that $\rho$ is surjective. To see that $\rho$ is injective, we note that since the reduced and universal norm on $C_c( \restrict{\widetilde E}{W})$ coincide they also coincide on $C_c( \restrict{\widetilde E}{W};\restrict{\cG}{W})$. Thus,  $\|f\|=\|f\|_\red=\|\pi^{\infty}(f)\|$ where 
 $\pi^\infty$ is as defined in \eqref{repn pi}. In order to compute $\pi^\infty(f)$, note that for 
 $\xi\in L^2(\restrict{\widetilde{E}}{W}\infty;\restrict{\cG}{W}\infty)=L^2(\widetilde{E}\infty;\{\infty\})$, we have
 \begin{align*}
    \|\xi \|^2
    &=
    \int_{\widetilde{E}} 
    | \xi(e)|^2\dd\sigma_{\infty}
    =
    \sum_{\gamma\in \infty\cG} \int_{\TT} | \xi(z\cdot \mathfrak{s}(\gamma))|^2\dd\lambda(z)
    \\&=
    \int_{\TT} | \xi(z\cdot \mathfrak{s}(\infty))|^2\dd\lambda(z)
    =
    \int_{\TT} | \xi(\infty_z)|^2\dd\lambda(z)
    =
    |\xi (\infty_{1})|^2.
 \end{align*}
    A computation analogous to that in \Cref{eq:rho multiplicative} then shows that 
    \[\| \pi^{\infty} (f)\xi\| =|(f*\xi)(\infty_{1})|\overset{\eqref{eq:rho multiplicative}}{=} |f(\infty_{1})|\, |\xi(\infty_{1})|= |\rho(f)|\, \|\xi \| ,
    \]
 so that $\|\pi^{\infty}(f)\| = |\rho(f)|$. It follows that $\rho$ is an isometric isomorphism of $\Cst( \restrict{\widetilde E}{W};\restrict{\cG}{W})=\Cst_\red( \restrict{\widetilde E}{W};\restrict{\cG}{W})$ onto $\CC$.

As mentioned previously, \Cref{reduced and full ses}  gives short exact sequences; by what we have shown so far, they can be rewritten to
\begin{gather*}
   0\longrightarrow\Cst( \restrict{\widetilde E}{U}; \restrict{\cG}{U})\stackrel{ i}{\longrightarrow}  \Cst(\widetilde E; \cG)\stackrel{ j}{\longrightarrow}  \CC\longrightarrow 0;\\
    0\longrightarrow\Cst_\red( \restrict{\widetilde E}{U}; \restrict{\cG}{U})\stackrel{
 i_{\red}}{\longrightarrow}  \Cst_\red(\widetilde E;\cG)\stackrel{
  j_{\red}}{\longrightarrow}  \CC\longrightarrow 0.
\end{gather*}
Since $ \restrict{\widetilde E}{U}=E$ and $\restrict{\cG}{U}=G$, this shows that $\Cst (E;G)$ and $\Cst_\red (E;G)$ are ideals in $\Cst (\widetilde{E}; \cG )$ and $\Cst_\red (\widetilde{E}; \cG )$, respectively, both of codimension $1$. This completes the proof. 
\end{proof}

Let $E=\TT\times G$ be the trivial twist, which as a groupoid is the Cartesian product. Then the full (or reduced) twisted groupoid $\Cst$-algebra is isomorphic to the full (or reduced) $\Cst$-algebra of $G$ and we obtain the following corollary of \Cref{twisted unitizations'}.

\begin{corollary}\label{lem:C*red for unitization'}
    Let~$G$ be a second-countable, \LCH\ and \etale\ groupoid with non-compact unit space, and let $\cG $ be its Alexandrov groupoid. Then $\Cst (\cG )\cong \Cst(G)^\sim$ and $\Cst_\red (\cG )\cong \Cst_\red (G)^\sim$.
\end{corollary}

\begin{remark}
Let $G$ be a locally compact, not necessarily Hausdorff, \etale\ groupoid with Hausdorff unit space. In \cite{Exel-Pitts:book}, Exel and Pitts have proposed that a certain  quotient of $\Cst_\red(G)$, called the {\em{essential groupoid $\Cst$-algebra}} and denoted by $\Cst_{\ess}(G)$, is the right replacement for the reduced $\Cst$-algebra of $G$  when $G$ is topologically principal.  Evidence towards this is that if $G$ is topologically principal, then (1) $C_0(G\z )$ detects ideals of $\Cst_{\ess}(G)$ and (2) $\Cst_{\ess}(G)$ is simple if and only $G$ is minimal  \cite[Theorem~22.6]{Exel-Pitts:book}. Again if  $G$ is topologically principal,  then the ideal  $I$  of $\Cst_\red(G)$ giving the quotient  $\Cst_{\ess}(G)$ consists of singular functions, and  $I$ vanishes when $G$ is Hausdorff \cite[Proposition~18.9]{Exel-Pitts:book}. The essential groupoid $\Cst$-algebras have been studied further, for example in  \cite{KM21, KKLRU}. In particular, there is a version of the 
Alevandrov groupoid
$\widetilde G$ for a non-Hausdorff $G$, and \cite[Proposition~7.8]{KKLRU} says that $\Cst_\ess (\cG )$ is isomorphic to $\Cst_\ess (G)^\sim$; applying this to the Hausdorff case, this gives another proof of \Cref{lem:C*red for unitization'} for the reduced groupoid $\Cst$-algebras.
\end{remark}

\subsection{Dynamic asymptotic dimension of the Alexandrov groupoid}\label{sect: dad for Alexandrov}

Finally, we are ready to show that dynamic asymptotic dimension is preserved by the Alexandrov construction. We begin with a brief digression which establishes, in particular, that the passage of finite nuclear dimension to quotients and ideals can be witnessed at the groupoid level with dynamic asymptotic dimension.

\begin{lemma}\label{lem:dad of restriction'}
Let~$G$ be a \LCH\ and \etale\ groupoid. Suppose that either
\begin{enumerate}[label=\textup{(\arabic*)}]
    \item\label{item:dad of restriction:open'} 
    $H = \restrict{G}{U}$ for some open set $U\subset G\z$,
    or
    \item $H$ is a closed subgroupoid of~$G$.
\end{enumerate}
If  $\dad(G) = d$, then $\dad(H) \leq d$.
\end{lemma}

\begin{remark}
    Before proving \Cref{lem:dad of restriction'}, we remark on its significance. Suppose that $\dad(G)$ and  $\dim(G\z )$ are finite, so that $\Cst_{\red}(G)$ is nuclear by \cite[Theorem 8.6]{GWY:2017:DAD}. In particular, $G$ is amenable by \cite[Corollary~6.2.14]{AnRe:Am-gps} and any closed subgroupoid of $G$ is amenable as well \cite[Proposition~5.11]{AnRe:Am-gps}. Let $U\subset G\z $ be an open invariant subset and let $W\coloneqq G\z \setminus U$. Then $G|_W$ is amenable, and inclusion and extension by zero and restriction induce a short exact sequence 
    \[0 \to \Cst_{\red}(\restrict{G}{U})\to \Cst_{\red}(G)\to \Cst_{\red}(\restrict{G}{W})\to 0\]
    by \cite[Proposition~5.2]{Williams:groupoid}. Then by \cite[Proposition~2.9]{WinterZacharias:nuclear},
    \[\max\{\dim_{\text{nuc}}(\Cst_{\red}(\restrict{G}{U})), \dim_{\text{nuc}}(\Cst_{\red}(\restrict{G}{W}))\}\leq \dim_{\text{nuc}}(\Cst_{\red}(G)).
    \]
   \Cref{lem:dad of restriction'} tells us that this $\Cst$-algebraic statement is already witnessed at the groupoid level by finite dynamic asymptotic dimension. 
\end{remark}

\begin{proof}[Proof of \Cref{lem:dad of restriction'}]
Suppose first that $H=\restrict{G}{U}$ for some open subset $U$ of the unit space. Let $V$ be an open and precompact subset of $H$. Since $s$ and $r$ are open and continuous maps, we can replace $V$ by $V \cup s(V) \cup r(V)$ if necessary, so assume without loss of generality that $s(V) \cup r(V) \subset V$.
Since $H$ is open, $V$ is also open in~$G$, and is precompact in~$G$ with $\cl[G]{V} = \cl[H]{V} \subset H$ by \Cref{lem:cpctness}~\ref{cpctness:item:precpct in Y}. Since $\cl[G]{V} = \cl[H]{V}$, we drop the superscript and simply write $\cl{V}$.

Since $\dad(G) = d$ there exist  $U_0, \ldots, U_d \subset G\z $ which are open in~$G$, cover $s(V) \cup r(V)$ and such that the subgroupoid \[G_i=\langle V\cap G|_{U_i}\rangle=\langle V\cap G|_{U_i\cap V}\rangle\] of~$G$  is precompact in~$G$. For $0 \leq i \leq d$, let $V_i = U_i \cap V$.  Then $\{V_i\}_{i=0}^d$ is an open  cover of $s(V) \cup r(V)$ in $H$. Notice  that $G_i=\langle V\cap G|_{V_i}\rangle$  is a subgroupoid of $H$. If $\alpha\in \cl[G]{G_{i}}$, then $s(\alpha)\in \cl[G]{V_{i}}$ by continuity of the source map. This implies that $s(\alpha)\in s(\cl{V}) \subset U$ because $V$ is a precompact subset of $H$. Similarly,  $r(\alpha)\in U$, and hence $\alpha\in H$. Thus, the closure of $G_i$ in $G$ is contained in $H$ and hence $G_{i}$ is precompact in $H$ by  \Cref{lem:cpctness}~\ref{cpctness:item:precpct in X vs Y}. This proves that $\dad(H)\leq d$.

\medskip
Now suppose that $H$ is a closed subgroupoid of~$G$. Let $V\subset H$ be a precompact open subset of $H$. Let $V'$ be open in~$G$ such that $V=H\cap V'$. By \Cref{lem:cpctness}~\ref{cpctness:item:precpct in Y}, $V$ is precompact in~$G$, and since~$G$ is locally compact, we can find a {\em precompact} open set $W$ containing $\cl[G]{V}$; replacing $V'$ by $V' \cap W$ if necessary we can assume without loss of generality that $V'$ is precompact.

We now apply $\dad(G)\leq d$ to $V'$ to get open sets $U'_{0},\ldots,U'_{d}$ of $G\z $ such that $s(V') \cup r(V')\subset \cup_{i=0}^{d} U'_{i}$ and such that $\langle \restrict{G}{U'_i}\cap V' \rangle$ is precompact. Then the collection of $U_{i}\coloneqq U'_{i}\cap H$ is an open cover in $H$ of $s(V)\cup r(V)$. Furthermore, $\langle \restrict{H}{U_i}\cap V \rangle$ is contained in $\langle \restrict{G}{U'_i}\cap V' \rangle$, and is therefore precompact in~$G$. Since $H$ is closed, \Cref{lem:cpctness}~\ref{cpctness:item:precpct in X vs Y} implies that $\langle \restrict{H}{U_i}\cap V \rangle$ is also precompact in $H$. As $V$ was arbitrary, it follows that $\dad(H)\leq d$.
\end{proof}

\begin{prop} \label{DAD of Gtilde'}
    Suppose that $G$ is a \LCH\ and \etale\ groupoid. Then $\dad (G)=\dad (\cG)$.
\end{prop}

\begin{proof}
Throughout this proof, we will repeatedly but tacitly use that $G$ is an open subspace of $\cG$ and that a precompact subset of $G$ is precompact in $\cG$ (see  \Cref{lem:cpctness}~\ref{cpctness:item:precpct in Y}); as the two (compact) closures coincide, we will refrain from using superscripts on closures in this case.

Suppose that $\dad (G)=d$. Fix an open and precompact $V\subset \cG$. Since $\cG\z $ is compact and open, we may assume that $\cG\z \subset V$ by possibly replacing $V$ with $V\cup \cG\z$. Set $W= V \cap (G \setminus G\z)$. Then $W$ is precompact and open in~$\cG$ because $G \setminus G\z $ is open in~$G$ and $V$ is precompact and open in~$\cG$. Since $\cG\z$ is open in $\cG$, the closure of $W$ in $\cG$ is contained in $\cG\setminus\cG\z=G\setminus G\z$; by \Cref{lem:cpctness}~\ref{cpctness:item:precpct in X vs Y}, $W$ is hence precompact in~$G$. Since $G$ is locally compact and Hausdorff and $G\setminus G\z$ is open in $G$, there exists an open precompact subset $W'$ of $G$ with $ \cl{W} \subset W'\subset \cl{W'}\subset  G\setminus G\z $.

Since $\dad(G)=d$, there exist open subsets $U_0, \ldots, U_d$ of $G\z $ that cover $s_G(W') \cup r_G(W')$
such that $\langle W' \cap \restrict{G}{U_i}\rangle$ is a precompact subgroupoid of~$G$. By replacing each $U_i$ with $U_i \cap (s(W') \cup r(W'))$, we can without loss of generality assume that the $U_i$'s are precompact in~$G$. Since $W\cap \restrict{G}{U_{i}}\subset \langle W' \cap \restrict{G}{U_i}\rangle$, we have that $G_{i}\coloneqq \langle W \cap \restrict{G}{U_i}\rangle$ is also a precompact subgroupoid of $G$ for each $0\leq i\leq d$. Since $\cl{W}$ is compact, we know that $s(\cl{W})\cup r(\cl{W}) =s_G(\cl{W})\cup r_G(\cl{W})$ is closed in $\cG\z$. Since $\cG\z $ is normal and $s(\cl{W})\cup r(\cl{W})\subset s(W') \cup r(W')\subset \bigcup_{i=0}^d U_i$, we can thus find an open set $U\subset \cG\z $ such that
 \[s(\cl{W})\cup r(\cl{W})\subset U\subset \cl[\cG\z ]{U}=\cl[\cG]{U} \subset \bigcup_{i=0}^d U_i.\]
 Note that since $\bigcup_{i=0}^d U_i\subset G\z $, $U$ contains no neighborhood of $\infty$ and is hence also open in~$G$. Since each $U_{i}$ is precompact in $G$, it follows that $U$ is precompact in $G$, so $\cl{U}\coloneqq \cl[G]{U}=\cl[\cG]{U}$. Now, if we let $U_0'\coloneqq U_0\cup [\cG\z \setminus \cl{U}]$, then $\{U_0',U_{1},\ldots, U_d\}$ is a cover of $\cG\z =s(V)\cup r(V)$ consisting of open subsets of $\cG\z $. Since $\cG\z \subset V$ and $W=V\setminus \cG\z $, we have for any $F\subset \cG\z $
\[ V \cap\restrict{\cG}{ F} = [W \cap \restrict{G}{F}]  \sqcup F,\]
and so $\langle V \cap\restrict{\cG}{ U_i}\rangle=G_i\cup U_i$. Moreover, since $U$ contains $s(W)\cup r(W)$, we have  \[[s(W)\cup r(W)] \cap U'_0 = [s(W)\cup r(W)] \cap U_0.\]
Hence, $W\cap \restrict{G}{U_0'}=W\cap \restrict{G}{U_0}$, and thus we likewise have $\langle V \cap\restrict{\cG}{ U_0'}\rangle=G_0 \cup U_0'$. Since each $G_{i}$ and $U_i$ is precompact in $G$, it follows that $G_0 \cup U_0'$ and each $G_i\cup U_i$ is precompact in $\cG$. 
 Hence, $U_0',U_{1},\ldots,U_d$ satisfies \Cref{defn-dad}, and so $\dad (\cG)\leq d= \dad (G)$. 
 
\medskip 

Now, suppose $\cG$ has DAD $d'$. Since $G\z $ is open in $\cG\z $ and $G = \restrict{\cG}{G\z }$ we can apply \Cref{lem:dad of restriction'}~\ref{item:dad of restriction:open'} to conclude that $\dad (G)\leq d'=\dad (\cG)$.
\end{proof}

\section{
Twisted groupoid \texorpdfstring{$\Cst$}{C*}-algebras over principal groupoids with finite dynamic asymptotic dimension} \label{twistedgroupoidCalgebras}
In this section we prove our main theorem: the generalization of  \cite[Theorem~8.6]{GWY:2017:DAD} from $\Cst$-algebras of groupoids to twisted groupoid $\Cst$-algebras.

\begin{thm}\label{holy grail}
Let $G$ be a second-countable, \LCH, principal and  \etale\ groupoid and  let $(E, \iota, \pi)$ be a twist over $G$. Suppose that $G\z$ has topological covering dimension $N$ and that $G$ has dynamic asymptotic 
dimension $d$. Then the nuclear dimension of $\Cst_\red(E;G)$ is at most  $(N+1)(d+1)-1$.
\end{thm}

By taking  $E$ to be the Cartesian product $\TT\times G$ we recover \cite[Theorem~8.6]{GWY:2017:DAD}. Before giving the proof of \Cref{holy grail} (below after \Cref{commutator-v2}), we  need to establish some preliminary results. 
In particular, our proof of \Cref{holy grail} is inspired by arguments in the proofs of  \cite[Theorem~8.6]{GWY:2017:DAD} and \cite[Theorem~6.2]{Kerr:dim}: both  pick up on a colored version of local subhomogeneity  \cite[Definition 1.5]{ENST:2020:Decomp}. The following \namecref{prop: colored loc subhom}  makes this explicit. In the proof of \Cref{holy grail}  we apply this \namecref{prop: colored loc subhom} to subhomogeneous $\Cst$-algebras $B_i$ that are the $\Cst$-algebras of the subgroupoids arising from applying the definition of finite asymptotic dimension.

\begin{prop}\label{prop: colored loc subhom}
Let $A$ be a unital $\Cst$-algebra and let $X\subset A$ be such that $\operatorname{span}(X)$ is dense in $A$. Let $d,N\in \NN$. Suppose that for every finite $\Ff\subset X$ and every $\varepsilon>0$  there exist $\Cst$-subalgebras $B_0,\ldots ,B_d$ of $A$ with \[\max_{0\leq i\leq d}\dim_{\nuc}(B_i)\leq N\] and $b_0,\ldots ,b_d\in A$ of norm at most $1$ such that $b_{i}\mathcal{F}b_{i}^*\subset B_{i}$ for $0\leq i\leq d$ and such that
\[\max_{x\in \mathcal{F}}\left\|x-\sum_{i=0}^d b_ixb_i^*\right\|<\varepsilon. \]
  Then the nuclear dimension of $A$ is at most $(d+1)(N+1)-1$.
\end{prop}

\begin{proof}
We will use the characterization of nuclear dimension from  \Cref{lemma: Same defns}. Since $\operatorname{span}(X)$ is dense in $A$, it suffices to show that the criteria from \Cref{lemma: Same defns} hold for any fixed finite subset $\Ff\subset X$ and $\varepsilon>0$.

By assumption, there exist $\Cst$-subalgebras $B_0,\ldots ,B_d\subset A$ with $\dim_{\nuc}(B_i)\leq N$ and $b_0,\ldots ,b_d\in A$ of norm at most $1$ such that 
$b_{i}\mathcal{F}b_{i}^*\subset B_{i}$ for $0\leq i\leq d$, and such that
\begin{equation}\label{eq:choice of b_i}
\left\|x-\sum_{i=0}^d b_ixb_i^*\right\|<\frac{\varepsilon}{2}\ \text{ for all } x\in \Ff. 
\end{equation}
Since each $B_i$ has nuclear dimension at most $ N $ and $b_i\Ff b_i^*$ is a finite subset of $B_i$, for   $0\leq k\leq  N $ there exist finite-dimensional $\Cst$-algebras  $F_{i,k}$ and \cpc\ maps $B_i\xrightarrow{\rho_{i,k}}F_{i,k}\xrightarrow{\phi_{i,k}} B_{i}$  with $\phi_{i,k}$ order zero such that for all $x\in \Ff$, 
\begin{equation}\label{eq:choice of psi and rho}
   \left\|\sum_{k=0}^N (\phi_{i,k}\circ \rho_{i,k})(b_i x b_i^*)-b_i x b_i^*\right\|<\frac{ \varepsilon}{2(d+1)}.
\end{equation}
Using Arveson's Extension Theorem, as stated in \Cref{thm: arv for us}, we extend $\rho_{i,k}\colon B_i\to F_{i,k}$ to \cpc\ maps $A\to F_{i,k}$, which we also denote by $\rho_{i,k}$. For $0\leq i\leq d$ and $0\leq k\leq N$, define $\psi_{i,k}\colon A\to F_{i,k}$ by $\psi_{i,k}(a)=\rho_{i,k}(b_iab_i^*)$ for all $a\in A$. Since $\|b_i\|\leq 1$ for each $i$, these are all \cpc\ maps. Thus, for $0\leq i\leq d$ and $0\leq k\leq  N $, we have \cpc\ maps 
\[A\xrightarrow{\psi_{i,k}} F_{i,k} \xrightarrow{\phi_{i,k}} B_{i}\subset A\]
with each $\phi_{i,k}$ order zero such that for each $x\in \Ff$,
\begin{align*}
    &
    \left\| \displaystyle{\sum_{i=0}^{d}}\displaystyle{\sum_{k=0}^{ N }}  (\phi_{i,k}\circ \psi_{i,k})(x)-x \right\|
    =\left\| \displaystyle{\sum_{i=0}^{d}}\displaystyle{\sum_{k=0}^{ N }}  (\phi_{i,k}\circ \rho_{i,k})(b_i x b_i^*)-x \right\|
    \\
    &\quad \leq \left\| \displaystyle{\sum_{i=0}^{d}}\displaystyle{\sum_{k=0}^{ N }} (\phi_{i,k}\circ \rho_{i,k})(b_i x b_i^*)-\displaystyle{\sum_{i=0}^{d}}b_i x b_i^*\right\| + \left\|\displaystyle{\sum_{i=0}^{d}} b_{i}xb_{i}^*-x\right\|
    \\
    &\quad<\left\| \displaystyle{\sum_{i=0}^{d}}\left(\displaystyle{\sum_{k=0}^{ N }} (\phi_{i,k}\circ \rho_{i,k})(b_i x b_i^*)-b_i x b_i^*\right)\right\| + \frac{\varepsilon}{2}
    &&  \text{(by \eqref{eq:choice of b_i})}
    \\
    &\quad\leq  (d+1) \max_{0\leq i\leq  N } \left\| \displaystyle{\sum_{k=0}^{ N }} (\phi_{i,k}\circ \rho_{i,k})(b_i x b_i^*)-b_i x b_i^* \right\| + \frac{\varepsilon}{2}
    &&
    \\
    &\quad< (d+1)\frac{\varepsilon}{2(d+1)} + \frac{\varepsilon}{2}= \varepsilon. 
    && \text{(by \eqref{eq:choice of psi and rho})}
\end{align*}

Thus, $\dim_\nuc(A)\leq (d+1)( N +1)-1$.
\end{proof}

Next, we examine the properties of the subgroupoids arising from the definition of dynamic asymptotic dimension: that they  are precompact groupoids translates to them having a finite open cover of bisections.

\begin{prop}\label{subgroupoid}   
Let $H$ be a second-countable, \LCH, principal  and  \etale\ groupoid and  let $(F, \iota, \pi)$ be a twist over $H$. Let $M\in \NN$ and suppose that $H$ has a finite open cover of $M$ bisections. 
\begin{enumerate}[label=\textup{(\arabic*)}]
\item\label{subgroupoid-1}  Then $H$ is amenable.
\item\label{subgroupoid-2} The primitive ideal space of $\Cst(F;H)$ is homeomorphic to the orbit space $H\z /H$ and each irreducible representation of $\Cst(F;H)$ has dimension at most $M$.
\item\label{subgroupoid-3} For $m\in\NN$, let $\Prim_m(\Cst(F;H))$  denote the set of primitive ideals of  irreducible representations of dimension $m$ and let $H_m\z =\{x\in H\z : |[x]|=m\}$. The homeomorphism of~\ref{subgroupoid-2} takes $H_m\z /H$ to $\Prim_m(\Cst(F;H))$.
\item\label{subgroupoid-4} Let $m\in \NN$. Then $H_m\z /H$ is locally compact and Hausdorff, and  $\dim(H_m\z )=\dim(H_m\z /H)$.
\item\label{subgroupoid-5} We have $\dim(H\z )=\max_{ 1\leq m \leq M}  \dim(H_m\z )$.
 \item\label{subgroupoid-6} 
 The decomposition rank of  $\Cst(F;H)$   is at most $\dim(H\z )$.
 \end{enumerate}
\end{prop}

\begin{proof} 
Let $x\in H\z $. Then $Hx$ meets each bisection at most once. Thus, $|Hx|\leq M$ for all $x\in H\z $. Hence, each orbit $[x]\coloneqq r(Hx)$ has at most $M$ elements. In particular, $[x]$ is closed in $H\z $ whence the orbit space $H\z /H$ is T$_{1}$. Since $H$ is principal, its stability subgroups are trivially abelian, and it follows from 
\cite[Lemma~5.4]{vanWyk-Williams}
that $H$ is amenable,  giving~\ref{subgroupoid-1}.

\medskip

Since $H$ is principal, it further follows from \cite[Proposition~3.3]{Clark-anHuef-RepTh} that $\Cst(F;H)$ is liminal. In particular, the spectrum and the primitive ideal space of $\Cst(F;H)$ are homeomorphic. 
By \cite[Theorem~3.4]{Clark-anHuef-RepTh}, the orbit space $H\z /H$ is homeomorphic to $\Prim(\Cst(F;H))$.  The homeomorphism is given by $[x]\mapsto\ker \pi^x$ where $\pi^x$ is the representation described at \eqref{repn pi}. But in the proof of \cite[Lemma~3.2]{Muhly-Williams-CtsTr2}, which assumes that $H$ is principal, Muhly and Williams show that $\pi^x$ is unitarily equivalent to a representation on~$\ell^2(Hx)$   which they denote by~$M^x$. Since $|Hx|\leq M$ it now follows that each irreducible representation of $\Cst(F;H)$ has dimension at most $M$ and  that the homeomorphism above takes $H_m\z /H$ to $\Prim_m(\Cst(F;H))$. This gives~\ref{subgroupoid-2} and~\ref{subgroupoid-3}.

\medskip

Let $m\in\NN$. By \cite[Proposition~3.6.4 and its proof]{Dixmier}, $\Prim_m(\Cst(F;H))$ is \LCH\ and hence  $H\z _m/H$ is \LCH\ by~\ref{subgroupoid-3}. 

If $\Prim_m(\Cst(F;H))=\emptyset$, then $H_m\z =H_m\z /H=\emptyset$ by~\ref{subgroupoid-3}, and their dimensions are both~$-1$. 
Suppose that there exists a largest $N_{1}$ in $\{1,\dots, M\}$ such that  $\Prim_{N_{1}}(\Cst(F;H))$ is non-empty (otherwise we are done with~\ref{subgroupoid-4}). By~\ref{subgroupoid-2}, $\Prim(\Cst(F;H))$ consists of exactly all irreducible representations of dimension at most $N_{1}$. Thus,  $\Prim_{N_{1}}(\Cst(F;H))$ is open in $\Prim(\Cst(F;H))$ by \cite[Proposition~3.6.3]{Dixmier}. Using~\ref{subgroupoid-3}, it follows that  $H_{N_{1}}\z /H$ is open in $H\z /H$.

Let $q\colon H\z \to H\z /H$ be the quotient map; it is open  by \cite[Proposition~2.12]{Williams:groupoid}. Write $U\coloneqq H_{N_{1}}\z $ and note that $U=q\inv (q(U))$. Thus, $U$ and $W\coloneqq H\z \setminus U$ are respectively, open and closed, and invariant subsets of $H\z$.  By \Cref{restriction of twist non-etale}, $F|_U$ is a twist over  $H|_U$ and $F|_W$ is a twist over  $H|_W$. By \cite[Lemma~3.1]{Clark-anHuef-RepTh} we get the exact sequence
\[
0\to \Cst( F|_U;  H|_U)\to \Cst(F;H)\to \Cst(F|_W;H|_W)\to 0.
\]
The ideal $\Cst(F|_U;H|_U)$ is homogeneous and~\ref{subgroupoid-2} implies it has primitive ideal space homeomorphic to $H_{N_{1}}\z /H=U/H$. Similarly, the quotient $\Cst(F|_W, H|_W)$ is subhomogeneous such that irreducible representations have dimension at most $N_{1}-1$  and has primitive ideal space homeomorphic to $W/H$.

Now consider the restriction $q_{N_{1}}\colon H_{N_{1}}\z =U\to U/H$ of the quotient map  $q$ to the open subset $U$. We now verify the hypotheses of \cite[Chapter~9, Proposition~2.16]{Pears:dim}. First, since $q$ is open and $U$ is open, so is $q_{N_{1}}$. Thus, $q_{N_{1}}$ is a continuous, open surjection.  Second, since $U$ and $U/H$ are \LCH\ and second-countable, they are both weakly paracompact and normal. Finally, $|q_{N_{1}}\inv ([y])|=N_{1}$ for all $[y]\in U/H$. Hence, it follows from \cite[Chapter~9, Proposition~2.16]{Pears:dim} that $\dim(H_{N_{1}}\z )=\dim(H_{N_{1}}\z /H)$.

If $\Prim_{m}(\Cst(F|_W, H|_W))=\emptyset$ for all $1\leq  m < N_{1}$, then we are done because the claim is trivial. Otherwise, let $N_{2}$ be the largest in $\{1, \dots, N_{1}-1\}$ such that $\Prim_{N_{2}}(\Cst(F|_W, H|_W))$ is non-empty. Again $\Prim_{N_{2}}(\Cst(F|_W;H|_W))$ is open in $\Prim(\Cst(F|_W;H|_W))$ by \cite[Proposition~3.6.3]{Dixmier}. Since $W$ is closed, $H|_W$ has the same topological properties as $H$. Also, for each $1\le m< N_{1}$, we have $H\z_{m}=(H|_{W})\z_{m}$. We may therefore replace $H$ by $H|_W$ in the above argument to get that $\dim(H_{N_{2}}\z )=\dim(H_{N_{2}}\z /H)$. Continue to get~\ref{subgroupoid-4}.
\medskip

Let $1\leq m\leq M$ and let $\Prim_{\leq m}(\Cst (F;H))$ be the space of primitive ideals of irreducible representations of dimension at most $m$ and set $H_{\leq m}\z =\{x\in H\z  : |[x]|\leq m\}$. From \cite[Proposition~3.6.3]{Dixmier} and~\ref{subgroupoid-2}, it follows that $H_{\leq m}\z /H$ is closed in $H\z /H$ and that $H_{m}\z /H$ is open in $H_{\leq m}\z /H$.  Thus, $H_{\leq m}\z =q\inv (H_{\leq m}\z /H)$ is closed in $H\z $ and  $H_m\z =q\inv (H_{m}\z /H)$ is open in $H_{\leq m}\z $. Since each $H_{\leq m}\z $ is a closed subspace of a second-countable, locally compact and Hausdorff space, it is normal. Now the normal space $H_{\leq m}\z $ is the disjoint union of the closed subset $H_{\leq m-1}\z $ and the open subset $H_m\z $, and \cite[Chapter~3, Corollary~5.8]{Pears:dim} gives
\[\dim(H_{\leq m}\z )\leq \max\{ \dim(H_m\z ),\ \dim(H_{\leq m-1}\z )\}.\]
It follows that
\begin{align*}
    \dim(H\z )=\dim(H_{\leq  M}\z )&\leq \max\{ \dim(H_M\z ),\ \dim(H_{\leq M-1}\z )\}\\
    & \leq\max \{ \dim(H_M\z ),\ \dim(H_{M-1}\z ),\ \dim(H_{\leq M-2}\z )\}\\
    &\ \vdots\\
    &\leq \max_{ 1\leq m \leq M} \dim(H_m\z ).
\end{align*}
To see that $\dim(H\z )=\max_{ 1\leq m \leq M}  \dim(H_m\z )$, we note that $H\z $ is metrizable, so that $\dim(H_m\z )\leq \dim(H\z )$ for each $m$ by \cite[Theorem~1.8.3]{Coornaert}. This gives \ref{subgroupoid-5}.

\medskip

For \ref{subgroupoid-6}, we apply \cite[Theorem~1.6]{Winter:decomposition-rank} to get that 
\begin{align*}
\dr(\Cst(F;H))&\leq \max_{ 1\leq m \leq M}\dim(\Prim_m(\Cst(F;H)))=\max_{1\leq m \leq M}\dim(H_m\z /H)\\
&= \max_{1\leq m \leq M}\dim(H_m\z )=\dim(H\z )
\end{align*}
using \ref{subgroupoid-4} and \ref{subgroupoid-5}. 
\end{proof}

\begin{remark}\label{rem-Phillips}
In the language of \cite[Theorem~2.16]{Phillips:recursive}
Proposition~\ref{subgroupoid} implies that $\Cst(F;H)$ has a recursive subhomogeneous decomposition with maximum matrix size at most $M$ and topological dimension at most $\dim(H\z )$.
\end{remark}

We need a version of \cite[Lemma~8.20]{GWY:2017:DAD}. Let $K$ be a compact subset of $E$.  We write $C_{K}(E;G)\coloneqq\{f\in C_c(E;G)\colon\supp f\subset K \}$. For $f\in C_c(E;G)$ and $h\in C_c(G\z )$ we define a ``commutator'' $[f,h]$ by 
\begin{align}
[f,h](e )=f(e )\,\bigl(h(s_E(e ))-h(r_E(e ))\bigr)\label{eq: comm}
\end{align}
for $e \in E.$ Notice that $[f,h]\in C_c(E;G)$.

\begin{lemma}\label{commutator-v2}
Let $G$ be a second-countable, \LCH\  and  \etale\ groupoid, and  let $(E, \iota, \pi)$ be a twist over $G$.
\begin{enumerate}[label=\textup{(\arabic*)}]
\item\label{commutator-v2-1}  Let $f\in C_c(E;G)$ and $g\in C_c(E)$ such that the point-wise product $f\cdot g$ is an element of $C_c(E;G)$ and $\pi(\supp g)$ is a bisection in $G$. Then
\[
\|f\cdot g\|_{\Cst_\red(E;G)}\leq\|f\|_{\Cst_\red(E;G)}\|g\|_\infty.
\]
\item\label{commutator-v2-2}  Let $\epsilon>0$, let $K$ be  a compact subset of $E$ and let $V$ be an open,  precompact neighborhood of $\pi(K)$ in $G$. There exists $\delta>0$ such that if $  h\in C_c(G\z )$ satisfies $\sup_{\gamma\in V}|  h(s_G(\gamma))-  h(r_G(\gamma))|<\delta$, then for any $f\in C_{K}(E;G)$,
\[
\|[f,  h]\|_{\Cst_\red(E;G)}\leq \epsilon\|f\|_{\Cst_\red(E;G)}.
\]
\end{enumerate}
\end{lemma}

\begin{proof}
Fix a unit $x\in G\z $ and let $\pi^x$ be the representation of $C_c(E;G)$ on $L^2(Ex;Gx)$ defined at \eqref{repn pi}. By definition of the norm on $\Cst_\red(E;G)$, it suffices to show that $\|\pi^x(f\cdot g)\|\leq\|f\|_{\Cst_\red(E;G)}\| g \|_\infty$ to conclude our first claim. Let  $\mathfrak{s}\colon G\to E$ denote a fixed  but arbitrary, not necessarily continuous, section of $\pi$.  We compute for $\xi\in C_c(E;G)\cap L^2(Ex;Gx)$ of norm at most one:
\[
\|\pi^x(f\cdot g )\xi\|^2=\|(f\cdot g )*\xi\|^2=\sum_{\gamma\in Gx}|((f\cdot g )*\xi)(\mathfrak{s}(\gamma))|^2.
\]
For each $\gamma\in Gx$ we have
\begin{align*}
    ((f\cdot g )*\xi)(\mathfrak{s}(\gamma))
    &=\sum_{\alpha\in r(\gamma)G}  (f\cdot g )(\mathfrak{s}(\alpha))\,\xi\bigl(\mathfrak{s}(\alpha)\inv \mathfrak{s}(\gamma)\bigr)\\
    &=\int_{E}(f\cdot g )(e)\,\xi\bigl(e\inv \mathfrak{s}(\gamma)\bigr)\, \dd\sigma^{r(\gamma)}(e)\\
    \intertext{by our choice of Haar system $\sigma$ on $E$. The change of variable $d=e\inv \mathfrak{s}(\gamma)$ yields}
    ((f\cdot g )*\xi)(\mathfrak{s}(\gamma))&=\int_{E}(f\cdot g )\bigl(\mathfrak{s}(\gamma)d\inv \bigr)\,\xi(d)\, \dd{\sigma}_{x}(d)\\
    &=\sum_{\alpha\in G x}(f\cdot g )\bigl(\mathfrak{s}(\gamma)\mathfrak{s}(\alpha)\inv \bigr)\,\xi(\mathfrak{s}(\alpha)).
\end{align*}
Suppose that $\mathfrak{s}(\gamma)\mathfrak{s}(\alpha)\inv \in\supp g $. Then $\alpha\in\pi(\supp g )\inv \gamma$.  Since  $\pi(\supp g )\inv $ is a bisection, it contains at most one element with source $r_G(\gamma)$. So for each $\gamma\in Gx$,  there is at most one  $\alpha_\gamma$ such that $\mathfrak{s}(\gamma)\mathfrak{s}(\alpha_\gamma)\inv \in\supp g$. Thus, 
\begin{align*}
    \|\pi^x(f\cdot  g )\xi\|^2
    &=\sum_{\stackrel{\gamma\in Gx}{\exists\alpha_\gamma\in \pi(\supp  g )\inv \gamma}}\bigl|
    (f\cdot  g )\bigl(\mathfrak{s}(\gamma)\mathfrak{s}(\alpha_\gamma)\inv \bigr)
    \,\xi(\mathfrak{s}(\alpha_\gamma))\bigr|^2\\
    &\leq \|f\|^2_\infty\|  g \|^2_\infty\sum_{\stackrel{\gamma\in Gx}{\exists\alpha_\gamma\in \pi(\supp  g )\inv \gamma}}|\xi(\mathfrak{s}(\alpha_\gamma))|^2.
    \end{align*}
We claim that $\alpha_{\gamma}$ for $\gamma$ ranging over $Gx$ is a non-repetitive subset of $Gx$. For if $\alpha_\gamma=\alpha_\delta$, then they are both in $\pi(\supp  g )\inv \gamma$ and in $\pi(\supp  g )\inv \delta$. Say $\alpha_\gamma=\xi\gamma=\eta\delta$ where $\xi,\eta\in \pi(\supp  g )\inv $. Since $\xi, \eta$ have the same range and $\pi(\supp  g )\inv $ is a bisection, we get that $\xi=\eta$, and then $\gamma=\delta$ as well, giving the claim. It follows that
   \[
    \sum_{\stackrel{\gamma\in Gx}{\exists\alpha_\gamma\in \pi(\supp  g )\inv \gamma}}|\xi(\mathfrak{s}(\alpha_\gamma))|^2 \le \sum_{\beta\in Gx}|\xi(\mathfrak{s}(\beta))|^2 = \|\xi\|^2.\]   
Thus,
\[ \|\pi^x(f\cdot  g )\xi\|^2 \leq \|f\|^2_{\Cst_\red(E;G)} \|g\|^2_\infty,\]
where we have used that $\|f\|_\infty\leq \|f\|_{\Cst_\red(E;G)}$ by \cite[Proposition~2.8]{BFRP:Twisted}.
This gives \ref{commutator-v2-1}.

\medskip

For \ref{commutator-v2-2}, 
let $L$ be the compact closure of $V$ in $G$. By Urysohn's lemma, there exists a  continuous function $b\colon G\to [0,1]$ with compact support in $L$ such that $b$ is $1$ on $\pi(K)$ and is $0$ on $G\setminus V$. 

Fix $  h\in C_c(G\z )$ and $f\in C_K(E;G)$. We set $  g \coloneqq (b\circ \pi)\cdot(  h\circ s_E-  h\circ r_E)$; then since $\pi\inv(L)$ is compact by Lemma~\ref{lem:pi proper},  $  g \in C_c(E)$ with compact support in $\pi\inv (L)$ and
\begin{equation}\label{eq norm of psi}
\|g\|_\infty\leq \sup_{\gamma\in V}|(  h\circ s_G-  h\circ r_G)(\gamma)|.
\end{equation}
 Since $  g (t\cdot e)=  g ( e)$ for all $e\in E$ and $t\in \TT$, we have $f\cdot  g \in C_c(E;G)$. Since the support of $f$ is contained in $K$ by assumption,  we have
\begin{align*}
[f,  h]( e)&=f( e)(  h(s_E( e))-  h(r_E( e))\\
&=\begin{cases}
f( e)b(\pi( e))(  h(s_E( e))-  h(r_E( e)) &\text{if $ e\in K$}\\
0&\text{else}
\end{cases}\\
&=\begin{cases}
f( e)  g ( e) &\text{if $ e\in K$}\\
0&\text{else}
\end{cases}\\
&=(f\cdot  g )( e).
\end{align*}
Since the support of $b$ is compact and  $G$ is \etale, there exist $n\in \NN\setminus\{0\}$ and $b_{1}, \dots, b_n\in C_c(G)$ such that the support of each $b_i$ is a bisection,   $b=\sum_{i=1}^nb_i$ and $\|b_i\|_\infty\leq \|b\|_\infty$. Then $g =\sum_{i=1}^n g_i$ where $g_i=(b_i\circ \pi)\cdot(  h\circ s_E-  h\circ r_E)$ and $\pi(\supp  g_i)$ is a bisection, and $\|  g_i\|_\infty\leq\|  g \|_\infty$.
Now, our above computation yields
\begin{align*}
\|[f,  h]\|_{\Cst_\red(E;G)}&=\|f\cdot  g \|_{\Cst_\red(E;G)}\\
&\leq \sum_{i=1}^n \|f\cdot  g_{i}\|_{\Cst_\red(E;G)}\\
&\leq \sum_{i=1}^n \|f\|_{\Cst_\red(E;G)}\|  g_i\|_\infty && \text{(using item \ref{commutator-v2-1})}\\
&\leq n\|f\|_{\Cst_\red(E;G)}\|  g \|_\infty\\
&\leq n\|f\|_{\Cst_\red(E;G)} \sup_{\gamma\in V}|(  h\circ s_G-  h\circ r_G)(\gamma)|
\end{align*}
using \eqref{eq norm of psi}. Since $b$ and hence $n$ depend only on $K$ and $V$, and  are independent of $  h$ and $f$, we conclude that $\delta=\epsilon/n$ works, giving~\ref{commutator-v2-2}. 
\end{proof}

\begin{remark}
By taking $E$ to be the Cartesian product $\TT\times G$, so that $\Cst_\red(G)\cong\Cst_\red(E;G)$,  \Cref{commutator-v2} makes it clear that \cite[Lemmas~8.19 and 8.20]{GWY:2017:DAD} do not require $G$ to be principal as assumed in  \cite{GWY:2017:DAD}.
\end{remark}

\begin{proof}[Proof of Theorem~\ref{holy grail}]  
First assume that the unit space $G\z $ is compact. Let $\Ff$ be a finite subset of $C_c(E;G)\setminus\{0\}$ and let $\varepsilon>0$. There exists a compact subset $K$ of $E$ such that $f\in \Ff$ implies $\supp f\subset K$. Since both $K\inv$ and $\pi\inv(\pi(K))$ are  compact sets (see \Cref{lem:pi proper}), we may assume that $K=K\inv $ and that $K=\pi\inv (\pi(K))$. Since $G\z =E\z $ is compact and open in~$E$, we may also assume that $G\z \subset K$ and hence that the characteristic function of $G\z$ is in $\Ff$, i.e., that $1\in\Ff$.

Let $V\subset G$ be an open and  precompact neighborhood of $\pi(K)$, and let $\delta$ be as in    Lemma~\ref{commutator-v2}~\ref{commutator-v2-2} for $\frac{\varepsilon}{(d+1)\max_{f\in \mathcal{F}}\|f\|_{\Cst_{\red}(E;G)}}$, $K$ and $V$.   Since  $G$ has  dynamic asymptotic dimension $d$, applying \cite[Proposition~7.1]{GWY:2017:DAD} to $\delta$ and the precompact, open subset $V$ of $G$ gives
    \begin{enumerate}[label=\textup{(\arabic*)}]
    \item\label{consequence1} open sets $U_0, \dots, U_d$ covering $G\z = r_G(V)\cup s_G(V)$ such that the subgroupoids $H_i$ generated by $\{\gamma\in V: s_G(\gamma), r_G(\gamma)\in U_i\}$ are open and  precompact in $G$ for $0\leq i\leq d$; 
    \item\label{consequence3}   continuous and   compactly supported    functions $ h_i  \colon G\z \to[0,1]$ with support in $U_i$ such that for $x\in G\z$ we have $\sum_{i=0}^d h_{i}(x)^2=1$, and for $\gamma\in V$ and $0\leq i\leq d$  we have 
    \[\sup_{\gamma\in V}|h_{i}(s_G(\gamma))- h_{i}(r_G(\gamma))|< \delta .\]
\end{enumerate}
Let $0\leq i\leq d$. Because $H_i$ is open, $\pi\inv (H_i)$ is a twist over~$H_i$ by \Cref{lem-right-up}. Since~$H_i$ is principal, \etale\ and precompact in~$G$, we may apply \Cref{subgroupoid}; it follows from Part~\ref{subgroupoid-6} that $\Cst_\red(\pi\inv (H_i);H_i)=\Cst(\pi\inv (H_i);H_i)$ has decomposition rank at most the covering dimension $N_i$ of $H_i\z $. Since $G$ is second-countable, \LCH, $G$ and hence $G\z $ is metrizable. Now  $H_i\z $ is a subspace of $G\z $, and it follows from \cite[Theorem~1.8.3]{Coornaert} that $N_i=\dim(H_i\z )\leq\dim(G\z )=N$. Thus, the decomposition rank and hence also the nuclear dimension of  $\Cst_\red(\pi\inv (H_i);H_i)$ is at most $N$. Using \Cref{lem-right-up}, we identify $\Cst_\red(\pi\inv (H_i);H_i)$ with a $\Cst$-subalgebra $B_i$ of $\Cst_\red(E;G)$.

Let $f_i\coloneqq f_{ h_{i}}\in C_c(E;G)$ be the image of $ h_{i}$ under the isomorphism described at \eqref{eq-diagonal} so that $\supp f_i\subset \iota(G\z \times \TT)$ and $f_i(\iota(x, z))=z h_{i}(x)$ for $(x,z)\in G\z \times \TT$. Then $f_i$ is self-adjoint since $ h_{i}$ is. We aim to show that $f_i$ and $B_i$ satisfy the hypotheses of \Cref{prop: colored loc subhom} with regard to the set $\Ff\subset X\coloneqq C_c(E;G)\setminus\{0\}$ and the $\varepsilon>0$ fixed above.

    Note first that $\|f_i\|_{\Cst_\red(\pi\inv (H_i);H_i)}= \| h_{i}\|_\infty\leq 1$. Next, we claim that $f_i * \Ff * f_i\subset   B_i$.  For $f\in C_c(E;G)$, $e\in E$, and $\mathfrak{s}\colon G\to E$ a (not necessarily continuous) section of $\pi$, we have 
\begin{align}
(f_i * f)( e)&=\sum_{\gamma\in r_E( e)G}f_i(\mathfrak{s}(\gamma))f(\mathfrak{s}(\gamma)\inv  e)\notag\\
&=\sum_{\substack{\gamma\in r_E( e)G:\\\mathfrak{s}(\gamma)\in \iota(G\z \times\TT)}}f_i(\mathfrak{s}(\gamma))f(\mathfrak{s}(\gamma)\inv  e)\notag\\
&=h_{i}(r_{E}( e))f( e);\label{eq:f_i * f}
\end{align}
similarly, $(f * f_i)( e)=f( e)\overline{ h_{i}(s_E( e))}=f( e) h_{i}(s_E( e))$ and 
\begin{equation*}
f_i * f * f_i=h_{i}(r_E( e))f( e)h_{i}(s_E( e)).
\end{equation*}
Thus, for $f\in \Ff\subset C_{K}(E;G)$ and $e\in E$,
\begin{align*}
(f_i * f * f_i)( e) \neq 0 &\implies  e\in K\text{ and } r_E( e), s_E( e)\in U_i \\
&\implies  \pi( e)\in V\text{ and } r_G(\pi( e)), s_G(\pi( e))\in U_i\\
&\implies  \pi( e)\in H_i.
\end{align*}
    This shows that $f_i  * f * f_i\in B_i$ and thus $f_i * \Ff * f_i \subset B_i$ as claimed.

    In order to apply \Cref{prop: colored loc subhom}, it remains to show that for each $f\in \mathcal{F}$,
\[\Big \|f-\sum_{i=0}^d f_i *f*f_i  \Big\|_{\Cst_\red(E;G)}\leq \varepsilon.\] 
Let $f\in \Ff\subset C_{K}(E;G)$. Item~\ref{consequence3} on page~\pageref{consequence3}, the choice of $\delta$ and \Cref{commutator-v2} imply that
\[
\|[f,h_i]\|_{\Cst_\red(E;G)} \leq \epsilon\|f\|_{\Cst_\red(E;G)} \leq \frac{\varepsilon}{d+1}.
\] 
An easy computation gives $f*f_i=f_i *f+ [f,  h_{i}]$. Also, using \Cref{eq:f_i * f} and that $\sum_{i=0}^d h_{i}(r_E(e))^2=1$ gives $(\sum_{i=0}^d f_{i}*f_i)*f=f$. Hence,
\[
\sum_{i=0}^d f_i *f*f_i =\sum_{i=0}^d f_i*\bigl( f_i* f+[f,  h_{i}]\bigr) =f+\sum_{i=0}^d  f_i*[f,  h_{i}].
\]
Thus
\begin{align*}
\Big \|f-\sum_{i=0}^d f_i *f*f_i  \Big\|_{\Cst_\red(E;G)}&=\Big\|\sum_{i=0}^d f_i*[f,  h_{i}]  \Big\|_{\Cst_\red(E;G)}\\
&\leq (d+1)\| h_{i}\|_\infty\|[f, h_{i}]\|_{\Cst_\red(E;G)} \\
& \leq \varepsilon .
\end{align*}
It now follows from  \Cref{prop: colored loc subhom} that if the unit space of $G$ is compact, then the nuclear dimension of $\Cst_\red(E;G)$ is at most $(d+1)(N+1)-1$.

Now we suppose that $G\z $ is not compact.  Let $\cG$ be the Alexandrov groupoid described in \Cref{lem:etale:tilde G:topologically'}, and let $(\widetilde E,\widetilde \iota, \widetilde\pi)$ be the Alexandrov twist over~$\cG$ described in \Cref{tilde E' v2}.  Since~$G$ has dynamic asymptotic dimension~$d$, so does $\cG$ by \Cref{DAD of Gtilde'}. By \Cref{dimension of unitization}   we have 
$\dim (\cG\z )=
\dim\big(G\z_{+}\big)
=\dim(G\z )\leq N$.
From the compact case we have that $\dim_\nuc(\Cst_\red(\widetilde E;\cG))$ is at most $(N+1)(d+1)-1$.  By \Cref{twisted unitizations'},  $\Cst_\red(\widetilde E;\cG)$ is the minimal unitization of $\Cst_\red(E;G)$, and hence it follows from  \cite[Remark~2.11]{WinterZacharias:nuclear}  that the nuclear dimension of $\Cst_\red(E;G)$ is also at most  $(N+1)(d+1)-1$. 
\end{proof}

\goodbreak
\section{Application to \texorpdfstring{$\Cst$}{C*}-algebras of non-principal groupoids}\label{sec-application}

In this section we consider a groupoid $G$ with potentially large stability subgroups. \Cref{non-principal groupoids} gives a bound on the nuclear dimension of $\Cst_\red(G)$ involving the topological dimension of $G\z $ and the dynamic asymptotic dimension of the quotient of $G$ by its isotropy subgroupoid. Its proof is based on an isomorphism from \cite{Clark-anHuef-RepTh} of  $\Cst(G)$ with a twisted groupoid  $\Cst$-algebra 
to which \Cref{holy grail} applies.

 The \emph{isotropy subgroupoid} $\Aa$ of $G$ is
\[\Aa\coloneqq\{\alpha\in  G  : r_{ G }(\alpha)=s_{ G }(\alpha)\};\] it is closed and hence locally compact, and acts on $ G $ 
on the right
via $\gamma\cdot\alpha=\gamma\alpha$ when $s_{ G }(\gamma)=r_{ G }(\alpha)$. The quotient $\Rr\coloneqq  G /\Aa$ is a principal groupoid.  We write $\pi \colon G\to\Rr$ for the quotient map.

Throughout this section, $G$ is a second-countable, \LCH\ groupoid, with abelian stability subgroups that vary continuously, i.e.,  the map \[u\mapsto \Aa_{u}\coloneqq\{\alpha\in G : r_G(\alpha)=u=s_G(\alpha)\}\] is continuous from $ G\z $ into the space of closed subgroups of $ G $ in the Fell topology (see \cite[\S3.4]{Williams:groupoid}). By \cite[Theorem~6.12]{Williams:groupoid}, this continuity is equivalent to $\Aa$ having a Haar system.
For \etale\ groupoids we can say more:

\begin{lemma}\label{cts stabilisers} Let $ G $ be a 
\LCH\  and \'etale groupoid. Then the stability subgroups of $G$ vary continuously if and only if $\Aa$ is open in $G$.
\end{lemma}

\begin{proof}
First, assume that $\Aa$ is open in $G$.  Since $G$ is \'etale it has a Haar system of counting measures, which, since $\Aa$ is open, restricts to a Haar system on $\Aa$. Now the  stability subgroups vary continuously by \cite[Theorem~6.12]{Williams:groupoid}.

Second, assume that the  stability subgroups of $G$  vary continuously. Then the restriction $r_{\Aa}$
of the range map  $r_G$ to $\Aa$ is open by   \cite[Theorem~6.12]{Williams:groupoid}.  Fix $\gamma\in \Aa$; we will find an open neighborhood $B$  in $G$ such that $\gamma\in B\subset \Aa$. 
 Since $G$ is \'etale, there exists an open bisection $C$ in $G$ with $\gamma\in C$. Set $B=C\cap\Aa$.   We claim that $B=C\cap r\inv_G (r_{\Aa}(B))$.  Since $C$ and $r\inv_{G} (r_{\Aa}(B))$ contain $B$ we have $B\subset C\cap r\inv_{G} (r_{\Aa}(B))$.
 If $\eta\in C\cap r_{G}\inv (r_{\Aa}(B))$, then  $r_{\Aa}(\eta)=r_{G}(\eta)\in r_{\Aa}(B)$, so there exists $\eta'\in B$ such that $r_{\Aa}(\eta)
 =r_{\Aa}(\eta')$. Since $\eta,\eta'\in C$ have the same range and $C$ is a bisection, we must have $\eta=\eta'\in B$, proving that $C\cap r_{G}\inv (r_{\Aa}(B))\subset B$. 
 Thus $B=C\cap r\inv_{G} (r_{\Aa}(B))$, as claimed.
 
 Since 
    $C$ is open in $G$, $B$ is open in $\Aa$. Since $r_{\Aa}\colon \Aa\to G\z$ is an open map, 
 $r_{\Aa}(B)$ is open in 
    $G\z$ and so
 it follows that $r\inv_{G} (r_{\Aa}(B))$ is open in $G$, and then so is $C\cap r\inv_{G} (r_{\Aa}(B))$. Now $B$ is open in $G$, and hence so is $\Aa$. 
\end{proof}

Let $\nu$ be a Haar system on $\Aa$ and denote
by $\widehat{\Aa}$ the spectrum of the separable commutative $\Cst$-algebra $\Cst (\Aa,\nu)$. Then $\widehat{\Aa}$ is a second-countable, \LCH\ space. As discussed on \cite[page~3630]{Muhly-Williams-Renault:tr3}, 
\[\widehat{\Aa}=\{(\chi,u)  : u\in  G\z , \chi\in\widehat{\Aa_u}\},\]
where $(\chi,u)(f)=\int_{\Aa_u} \chi(a)f(a) d\nu^u(a)$ for $f\in C_c(\Aa)$.

We write $p\colon \widehat{\Aa}\to G\z =\Rr\z$ for the map $(\chi,u)\mapsto u$, which is open by \cite[page~3631]{Muhly-Williams-Renault:tr3}. If $(\chi,u) \in\widehat{\Aa}$ 
and $\gamma\in  G $ satisfy $
u
=r_{ G }(\gamma)$, then there is a  character  $\chi\cdot\gamma\in\widehat{\Aa}_{s(\gamma)}$ defined by 
$\chi\cdot\gamma(a)=\chi(\gamma a\gamma\inv)$.
Since $\chi\cdot\gamma$ depends only on $\dot{\gamma}\coloneqq \pi(\gamma)$,  there is a right action of $\Rr$ on $\widehat{\Aa}$ given by
\begin{equation}\label{Raction}
(\chi,u)\cdot\dot{\gamma}=(\chi\cdot\gamma, s_{ G }(\gamma ))
\end{equation}
for $\gamma\in G $ with $r_{ G }(\gamma)=u$.

We will now consider the transformation groupoid $\widehat{\Aa}\rtimes\Rr$ of this action as in \cite{Muhly-Williams-Renault:tr3}: as a set, it is given by
\begin{equation*}\label{MRWgroupoid}
\widehat{\Aa}*\Rr=\{(\chi, u,\dot{\gamma})\in\widehat{\Aa}\times\Rr : u=r_{ G }(\gamma)\}.
\end{equation*}
Two elements $(\chi,u,\dot{\gamma})$ and $(\eta,v, \dot\alpha)$ of $\widehat{\Aa}*\Rr$ are composable if and only  if $(\eta, v)=(\chi,u)\cdot\gamma=(\chi\cdot\gamma, s_{ G }(\gamma))$, and then
\[(\chi,u, \dot{\gamma})(\chi\cdot\gamma, s_{ G }(\gamma),\dot{\alpha})=(\chi,u, \dot{\gamma}\dot{\alpha});\]
the inverse is defined by $(\chi,u, \dot{\gamma})\inv =(\chi\cdot\gamma,s_{ G }(\gamma),\dot{\gamma}\inv )$. 
The range of $(\chi,u, \dot{\gamma})$ is hence given by $(\chi , u, u)$ and its
source by $(\chi \cdot \gamma , s_{ G }(\gamma ), s_{ G }(\gamma ))$, and 
the unit space of $\widehat{\Aa}\rtimes \Rr$ is $\{(\chi, u, u) : (\chi, u) \in\widehat{\Aa}\}$.

To simplify notation, we will write the elements of $\widehat{\Aa}\rtimes\Rr$ as $(\chi,\dot\gamma)$, and then  the unit space is identified with $\widehat{\Aa}$. Finally, observe that $\widehat{\Aa}*\Rr$ is closed in $\widehat{\Aa}\times\Rr$, and so $\widehat{\Aa}\rtimes\Rr$ is locally compact and Hausdorff.  It is straightforward to verify that   $\widehat{\Aa}\rtimes\Rr$ is principal.

\begin{lemma}\label{lem:summary}
Let $ G $ be a second-countable, locally compact and Hausdorff  groupoid
(not necessarily \etale)
with abelian and continuously varying stability subgroups. Let $\Aa$ be the isotropy groupoid and let $\Rr= G /\Aa$  be the quotient groupoid.

\begin{enumerate}[label=\textup{(\arabic*)}]

\item\label{it:cR T2}
Then $\Rr$ is locally compact and  Hausdorff.

\item\label{it:cA open} The subgroupoid  $\Aa$ is open in $ G $ if and only if  $\Rr $ is r-discrete.

\item\label{it:implies cR etale} 
If $ G $ is \etale, then $\Rr$ is \etale.

\end{enumerate}
\end{lemma}

\begin{proof}[Proof of Lemma~\ref{lem:summary}]
For \ref{it:cR T2}, the action of $ G $ on itself is proper by \cite[Example~2.16]{Williams:groupoid}. Since $\Aa$ is closed in $ G $, it follows easily from \cite[Proposition~2.17]{Williams:groupoid} that the action of $\Aa$ on $ G $ is  proper. Moreover, $\Aa$ is itself locally compact and Hausdorff. Since the stability subgroups vary continuously, 
the restriction of the range map to $\Aa$ is open  by \cite[Theorem~6.12]{Williams:groupoid}. It now follows from \cite[Proposition~2.18]{Williams:groupoid} applied to $\Aa$ that $\Rr= G /\Aa$ is locally compact and Hausdorff.

    \medskip
    
    For \ref{it:cA open}, we observe that $\pi\inv(\Rr\z )=\Aa$ because
    \begin{align*}
        \gamma \in\pi\inv(\Rr\z ) 
        &\iff  \pi(\gamma ) = \pi(\gamma )\pi(\gamma )\inv 
        \iff  \pi(\gamma ) = \pi(r(\gamma ))\\
        &\iff \exists a\in \Aa, \gamma  a = r(\gamma ) \iff \gamma \in\Aa.
    \end{align*}
   Now $\Rr$ is r-discrete if and only if $\Rr\z $ is open, and  by definition of the quotient topology, this happens  if and only if $\pi\inv (\Rr\z )=\Aa$ is open in $ G $.
    
    \medskip
    
    For \ref{it:implies cR etale}, suppose that $G$ is \etale. Since the stability subgroups vary continuously, $\Aa$ is open by \Cref{cts stabilisers}.
    Then $\Rr$ is r-discrete by  Part~\ref{it:cA open}. For $U\subset \Rr$, one checks that
    \begin{align*}
        \pi\inv(r_{\Rr}(U))=\Aa\cap r_{ G }\inv \Bigl(r_{ G }\bigl(\pi\inv (U)\bigr)\Bigr).
    \end{align*}
    Since $G$ is \etale,  $r_{ G }$ is an open map. So if $U$ is open in $\Rr$, then $\pi\inv(r_{\Rr}(U))$ is open in $ G $, using
        the above equality and
    that $\Aa$ is open. Since $\Rr\z$ is open and $\Rr$ has the quotient topology, this proves that $r_{\Rr}\colon \Rr\to\Rr\z  $ is  open.
     Thus $\Rr$ is \etale\ by \cite[Proposition~1.29]{Williams:groupoid}.
  \end{proof}

In the following, we consider a general transformation groupoid rather than the specific $\widehat{\Aa}\rtimes \Rr$; we thank the referee for bringing to our attention that our proofs work in that generality.
\Cref{it:dad K vs dad XrtimesK} of the next proposition was subsequently proved in greater generality in \cite[Proposition~2.4]{Bo:2023:DAD-pp}.

\begin{prop}
\label{prop finite dad for ANY transformation groupoid}
Let $X$ be a \LCH\ space and let $\mathcal{K}$ be a  \LCH\  and  \etale\ groupoid. Suppose that $\mathcal{K}$ acts continuously on the right of $X$ with surjective anchor map $\sigma\colon X\to \mathcal{K}\z$. 

\begin{enumerate}[label=\textup{(\arabic*)}]
    \item\label{it:cR etale implies cH etale NEW} The transformation groupoid $X\rtimes\mathcal{K}$ is \etale.
    \item\label{it:pr cts open NEW} 
    If $\sigma$ is open, then the projection map  $\mathrm{pr}_{\mathcal{K}}\colon X\rtimes \mathcal{K}\to \mathcal{K}$, $(x,k)\mapsto k$,  is   continuous and open.
    \item\label{it:dad K vs dad XrtimesK}
    If $\sigma$ is open and if $\dad(\mathcal{K})\leq d$, then  $\dad(X\rtimes\mathcal{K})\leq d$.
\end{enumerate}
\end{prop}

\begin{proof}
We set $\Hh\coloneqq  X \rtimes \mathcal{K} $ 
    for convenience.

\medskip

    For \ref{it:cR etale implies cH etale NEW}, $ X  \times \mathcal{K}\z $ is open in $ X  \times \mathcal{K}$, so that
    \[
        \Hh\z  = ( X  \times \mathcal{K}\z )\cap \Hh
    \]
    is open in $\Hh$, proving that $\Hh$ is r-discrete. We henceforth identify $\Hh\z$ with $X$ via the map $(x,\sigma(x))\mapsto x$.
    To see that $\Hh$ is \etale, we will show that the range map $r_{\Hh}\colon \Hh\to\Hh\z $ is open. Let $\net{x_\lambda}{\lambda}$ be a net in $\Hh\z$ that converges to  $r_{\Hh}(x,k)=x$ in $\Hh\z$.
    Then $\sigma(x_\lambda)\to \sigma(x)=r_{\mathcal{K}}(k)$. Since $\mathcal{K}$ is \etale, $r_{\mathcal{K}}$ is open by \cite[Proposition~1.29]{Williams:groupoid}. It follows from Fell's criterion (\cite[Proposition~1.1]{Williams:groupoid}) that there exists a subnet $\net{x_{f(\mu)}}{\mu}$ of $\net{x_\lambda}{\lambda}$ and a net $\net{k_\mu}{\mu}$ in $\mathcal{K}$ such that $r_{\mathcal{K}}(k_\mu)=\sigma(x_{f(\mu)})$ and $k_\mu\to k$. Thus, $(x_{f(\mu)},k_\mu)$ is a lift of  
    $x_{f(\mu)}\in \Hh\z$
    in $\Hh$ under $r_{\Hh}$ which converges to $(x,k)$. By another application of Fell's criterion, this implies that $r_{\Hh}$ is open. Thus, $\Hh$ is \etale.

\medskip

For \ref{it:pr cts open NEW}, we first observe that continuity of $\mathrm{pr}_{\mathcal{K}}$ is immediate because $ X \rtimes\mathcal{K}$ has the subspace topology  from $ X \times\mathcal{K}$.
To see that $\mathrm{pr}_{\mathcal{K}}$ is open, we will use Fell's criterion  twice.
Let $\net{ k_{\lambda}}{\lambda}$ be a net in $\mathcal{K}$ that converges to $ k=\mathrm{pr}_{\mathcal{K}} ( x , k)$. Then $u_\lambda\coloneqq r_{\mathcal{K}}(k_{\lambda})$ is a net converging to $u\coloneqq r_{\mathcal{K}}(k)=\sigma( x )$. Since $\sigma$ is open and surjective, by Fell's criterion there exists a subnet  $\net{u_{f(\mu)}}{\mu}$ of $\net{u_{\lambda}}{\lambda}$ and a lift $\net{ x _{\mu}}{\mu}$ of that subnet under~$\sigma$ (i.e., $\sigma( x_\mu)=u_{f(\mu)}$) such that 
$x_\mu\to x $. Then  $\net{( x_\mu,  k_{f(\mu)})}{\mu}$ is a net in $ X \rtimes\mathcal{K}$ that is a lift of the subnet $\net{ k_{f(\mu)}}{\mu}$ under $\mathrm{pr}_{\mathcal{K}}$ and that converges to $( x ,  k)$. By Fell's criterion, this proves that $\mathrm{pr}_{\mathcal{K}}$ is open.

\medskip

   For \ref{it:dad K vs dad XrtimesK},  let $W\subset \Hh$ be precompact and  open; without loss of generality assume $W=W\inv$, so  that $r_{\Hh}(W)\cup s_{\Hh}(W)=r_{\Hh}(W)$. By Part~\ref{it:pr cts open NEW}, the projection map $\mathrm{pr}_{\mathcal{K}}\colon \Hh\to \mathcal{K}$ onto the second component is a continuous open map, so that
    $W'\coloneqq   \mathrm{pr}_{\mathcal{K}} (W)= (W')\inv$ is open in $ \mathcal{K} $. Recall that images of precompact sets under continuous maps are precompact, so 
    long as the codomain is Hausdorff.
    Therefore, $W'$ is precompact.   
    Since $\dad( \mathcal{K} )\leq d$, there  exist open $U_{0},\ldots , U_{d} \subset  \mathcal{K}\z $ which cover $r_{ \mathcal{K} }(W')$ and such that $$ \mathcal{K} _{i}\coloneqq   \langle W'\cap \restrict{ \mathcal{K} }{U_{i}} \rangle $$ is precompact in $ \mathcal{K} $.  Let $V_{i}\coloneqq  \sigma\inv (U_{i})\cap r_{\Hh}(W)  \subset  
    X$. Since $\sigma $ is continuous, the $V_{i}$'s are open subsets of $ X $. We claim that since the $U_{i}$'s cover $r_{ \mathcal{K} }(W')$, the $V_{i}$'s cover $r_{\Hh}(W)$. For  an arbitrary $x\in r_{\Hh}(W)$, take $k\in \mathcal{K} $ such that  $(x,k)\in W$; in particular, 
    \(
    \sigma(x)=r_{ \mathcal{K} }(k)
    \)
     and $k=\mathrm{pr}_{\mathcal{K}}(x,k)\in W'$. Thus $\sigma (x)=r_{ \mathcal{K} }(k)\in U_{i}$ for some $0\leq i\leq d$, i.e., $x\in \sigma\inv (U_{i}) \cap r_{\Hh}(W)= V_{i}$, as claimed.

   To conclude that $\dad(\Hh)\leq d$, it now suffices to show that $\Hh_{i}\coloneqq   \langle W\cap \restrict{\Hh}{V_{i}}\rangle $ is precompact. Towards this, note first that each $V_i$ is precompact. Indeed, since $ X $ is Hausdorff,  $W$ is precompact and $r_{\Hh}$ is
   continuous, it follows, as argued before, that $r_{\Hh}(W)$ and hence $V_{i}$ is precompact.
    
    We claim that $\Hh_{i}$ is contained in $ X  \times  \mathcal{K} _{i}$. To see this, let
    $(x,k)\in \Hh_{i}$. 
    There exist finitely many $( x_{j}, k_{j})\in W\cap \restrict{\Hh}{V_{i}}$ such that $( x, k)=( x_{1}, k_{1})\cdots  ( x_{m}, k_{m})$. As $( x_{j}, k_{j})\in W$, we know that $ k_{j}\in W'$. As $( x_{j}, k_{j})\in  \restrict{\Hh}{V_{i}}$, we have that both $x_{j} =r_{\Hh}( x_{j}, k_{j})$ and $x_{j}\cdot  k_{j} = s_{\Hh}( x_{j}, k_{j})$  is in $V_{i}$; in particular, $r_{ \mathcal{K} }( k_{j})=\sigma ( x_{j})$ and $s_{ \mathcal{K} }( k_{j})=\sigma ( x_{j}\cdot  k_{j})$ are both elements of $\sigma (V_{i})\subset U_{i}$. Thus, $ k_{j}\in  \mathcal{K} |_{U_{i}}$ and so $ k= k_{1}\cdots k_{m}\in  \mathcal{K} _{i}$, as claimed.

    To see that $\Hh_{i}$ is precompact, suppose $\net{( x_{\lambda}, k_\lambda)}{\lambda}$ is a net in $\Hh_{i}$; we must show that there exists a subnet that converges (in $\Hh$, not necessarily in $\Hh_{i}$). By the above argument, $ k_{\lambda}$ is in the precompact  $ \mathcal{K} _{i}$, so by passing to a subnet, we can without loss of generality assume that $\net{ k_{\lambda}}{\lambda}$ converges (again, in $ \mathcal{K} $ and not necessarily in $ \mathcal{K} _{i}$). By definition of $\Hh_{i}$, we further have that all  $x_\lambda=r_{\Hh}( x_{\lambda}, k_\lambda)$ are elements of the precompact $V_{i}$. Again, passing to a subnet allows us  without loss of generality to assume that $\net{x_\lambda}{\lambda}$
    converges in~$ X $. Since $\Hh$ has the subspace topology inherited from $ X \times \mathcal{K} $, we conclude that $\net{( x_\lambda, k_\lambda)}{\lambda}$ (has a subnet that) converges in $\Hh$, as required. This proves that $\Hh_{i}$ is precompact. 
    Thus, $\dad(\Hh)\leq d$.
\end{proof}

\begin{corollary}\label{cor finite dad for transformation groupoid}
Let $G$ be a second-countable, \LCH\ and \etale\ groupoid. Suppose that the stability subgroups of~$G$ are abelian 
and vary continuously.
Let $\Rr$ be the quotient groupoid of $ G $ by the isotropy groupoid $\Aa$. The transformation groupoid $\widehat{\Aa}\rtimes \Rr$ is \etale. Moreover,
if $\dad(\Rr)\leq d$, then  $\dad(\widehat{\Aa}\rtimes\Rr)\leq d$.
\end{corollary}

\begin{proof}
    Since the  abelian stability subgroups of $G$ vary continuously, we may invoke Lemma~\ref{lem:summary}: $\Rr$ is locally compact and Hausdorff by Part~\ref{it:cR T2}, and it is \etale\ by Part~\ref{it:implies cR etale} since $G$ is \etale.
    Since $\widehat{\Aa}$ is the spectrum of the commutative  $\Cst$-algebra $\Cst(\Aa)$, it is locally compact and Hausdorff. Since the anchor map $p\colon \widehat{\Aa}\to \Rr\z$ is open (\cite[page~3631]{Muhly-Williams-Renault:tr3}), we may thus further invoke  Proposition~\ref{prop finite dad for ANY transformation groupoid}: $\widehat{\Aa}\rtimes \Rr$ is \etale\ by  Part~\ref{it:cR etale implies cH etale NEW}, and  $\dad(\widehat{\Aa}\rtimes\Rr)\leq \dad(\Rr)$ by Part~\ref{it:dad K vs dad XrtimesK}.
\end{proof}

\begin{corollary}
\label{non-principal groupoids}
Let $ G $ be a second-countable, \LCH\ and \etale\ groupoid.  Suppose that the orbits of $ G $ are closed in $ G $ and that the stability subgroups of $ G $ are abelian 
and vary continuously.
Let $\Aa$ be the isotropy groupoid, 
let $\widehat{\Aa} $ be the spectrum of the commutative $\Cst$-algebra $\Cst(\Aa)$, and let $\Rr$ be the quotient groupoid of $ G $ by  $\Aa$.
Suppose
that the  topological dimension of $\widehat{\Aa} $ is at most $N$ and that $\Rr$ has finite dynamic asymptotic dimension at most $d$.  Then the nuclear dimension of $\Cst_\red( G )=\Cst( G )$ is at most $(N+1)(d+1)-1$.
\end{corollary}

\begin{proof} 
Let  $\Dd$ be the  $\TT$-groupoid over $\widehat{\Aa} \rtimes\Rr$ constructed in \cite[\S4, page~3636]{Muhly-Williams-Renault:tr3}. Since 
the abelian stability subgroups vary continuously and $G$ is \etale, so are $\Rr$ and $\widehat{\Aa} \rtimes\Rr$ by \Cref{lem:summary}.
Since $\Rr$ has finite dynamic asymptotic dimension at most $d$, so does $\widehat{\Aa} \rtimes\Rr$ by 
\Cref{cor finite dad for transformation groupoid}.
By \Cref{holy grail}, $\Cst_\red(\Dd;\widehat{\Aa} \rtimes\Rr)$ has nuclear dimension at most $(N+1)(d+1)-1$.  

Since $\widehat{\Aa} \rtimes\Rr$ is principal and has finite dynamic asymptotic dimension, it is amenable by \cite[Corollary~8.25]{GWY:2017:DAD}. Since  $\Dd$ is an extension of $\widehat{\Aa} \rtimes\Rr$ by $\TT$, it follows that $\Dd$ is topologically amenable  by \cite[Proposition~5.1.2]{AnRe:Am-gps}. Hence, $\Cst_\red(\Dd)=\Cst(\Dd)$, and then also $\Cst_\red(\Dd;\widehat{\Aa} \rtimes\Rr)=\Cst(\Dd;\widehat{\Aa} \rtimes\Rr)$. 

Since the stability subgroups are abelian  and the orbits are closed, $G$ is amenable by \cite[Lemma~5.4]{vanWyk-Williams}, and hence $\Cst( G )=\Cst_\red(G)$. 
Since  the stability subgroups also vary continuously,
$\Cst(\Dd;\widehat{\Aa} \rtimes\Rr)$ and $\Cst( G )$ are isomorphic by \cite[Proposition~6.3]{Clark-anHuef-RepTh}. Now the \namecref{non-principal groupoids} follows.
\end{proof}
We now give two examples which illustrate  \Cref{non-principal groupoids}.

\begin{example}\label{ex: transformation} In \cite[Example~5.4]{Williams:cts_trace_transformation_groups} there is an action of $\RR$ on $\CC$ which is $\sigma$-proper (that is, proper relative to the stability subgroups); the orbits are closed and the stability subgroups vary continuously.  In the following, we restrict this action to $\ZZ$ and show that this gives an example of an \etale\ groupoid $ G $ which satisfies the hypotheses of \Cref{non-principal groupoids}.

Set \[X=\{z\in\CC : |z|\in\NN\}=\{ 0 \}\cup \bigsqcup_{n\geq 1}n\TT;\]
then $\ZZ$ acts on $X$ via
\[n\cdot z=
\begin{cases}
 0  & \text{ if } z= 0 
\\
e^{2\pi i (n/|z|)}z & \text{ if } z\neq  0. 
\end{cases}
\]
The orbits are concentric circles about the origin and the stability subgroups are
\[
S_z\coloneqq\{n\in \ZZ:n\cdot z=z\}
=
\begin{cases}
\ZZ&\text{if $z=0$}
\\
|z|\ZZ&\text{if $z\neq 0$.}\end{cases}
\]

Let $ G =\ZZ\ltimes X$ be the transformation-group groupoid,  which is \etale\ since $\ZZ$ is discrete. Then the isotropy groupoid is 
\[\Aa=\big(\ZZ\times \{ 0 \}\big)\cup\bigsqcup_{n\geq 1} \big(n\ZZ\times n\TT\big)\]
which is open in $\ZZ\times X$, i.e., open in $G$. Thus, $\Rr\coloneqq G /\Aa$ is \etale\ by Lemma~\ref{lem:summary}\ref{it:implies cR etale}. For $(n,z)\in G $ write $[(n,z)]$  for the equivalence class in  $\Rr$. Note that if $(k,w)\in\Aa$, then the product $(n,z)(k,w)$ is defined  if and only if $z=k\cdot w=w$. We have $[(n,0)]=\ZZ\times\{0\}$ and if $z\neq 0$ we have
\begin{align*}
    [(n,z)]&=\{(n, z)(m|z|, z) : m\in\ZZ\}
    =(n+|z|\ZZ)\times\{z\}.
\end{align*}
In particular,
$(n,z)\sim (n+|z|m, z)$ for every $m\in\ZZ$.

We claim that $\Rr$ is a proper groupoid, and hence $\dad(\Rr)=0$.
Fix a compact subset $K$ of the unit space $\Rr\z $ (identified  with $X$). 
Without loss of generality we can make $K$ larger, and so we can assume that $K=\{z\in X : |z|\leq r\}$ for some $r\in \NN$.  Since $\ZZ\cdot K=K$ we have 
\[
\restrict{\Rr}{K}=\{[(n,z)] : r([(n,z)]), s([(n,z)])\in K\}=\{[(n,z)] : z\in K\}.
\]
Consider a sequence $[(n_i,z_i)]$  in $\restrict{\Rr}{K}$; we will show that it admits a subsequence that converges. Since $K$ is compact, we may assume that $z_i\to z$ for some $z\in K$. Then  $|z_i|=|z|$ eventually. For each $n_i$, choose $m_i\in \ZZ$ such that $n_i+|z|m_i$ is in $\{0,\dots, |z|-1\}$.
Then $[(n_i,z_i)]=[(n_i+|z|m_i,z_i)]$ has a subsequence in which the first component is constant; thus, we may without loss of generality assume that our sequence is of the form $[(n,z_i)]$ for some $n\in \{0,\dots, |z|-1\}$. Since $(n,z_i)\to (n,z)$ in $ G $, this sequence converges to $[(n,z)]$ in $\restrict{\Rr}{K}$.  
Thus, $\restrict{\Rr}{K}$ is compact. It follows that $\Rr$ is a proper groupoid.
In particular, if $V\subset \Rr$ is precompact open, then since the unit space of $\Rr$ is open, $U_{0}\coloneqq r(V)\cup s(V)$ is precompact, open in $\Rr\z $. Now choose a compact $K$ large enough to contain the closure of $U_{0}$. Then $\langle V\cap \restrict{\Rr}{U_{0}}\rangle$ is contained in $\restrict{\Rr}{K}$, which is compact by the above argument, and so it follows that $\dad(\Rr)=0$.

We claim that the topological dimension of $\widehat{\Aa} $ is $1$.
Since $\widehat{\Aa} $ is the spectrum of a separable, commutative $\Cst$-algebra, it is \LCH\ and second-countable, and hence is normal.
Since the stability subgroups are constant along orbits, we have
\[
\widehat{\Aa} 
=\{(\chi, z) : z\in X, \chi\in\widehat{S}_{z}
\}
=\bigcup_{n\in  \NN}\widehat{S}_{n}
\times n\TT,
\]
which is a countable union of closed subsets of $\widehat{\Aa} $. 
By the countable-union theorem \cite[Theorem~1.7.1]{Coornaert}, $\dim \widehat{\Aa}$ is $\sup_{n\in\NN}\dim \big(\widehat{S}_{n}\times n\TT\big)$. The dimension of   $\widehat{S}_{0}\times 0\TT=\TT\times\{ 0 \}$ is $1$.
Each $\widehat{S}_{n}\times n\TT$ for $n\neq 0$ is homeomorphic to $\TT\times\TT$, which has topological dimension $1$ by writing it as the union $\TT\times\{1\}\cup\{1\}\times\TT$ of closed subsets of dimension $1$.
Thus, $\dim(\widehat{\Aa} )$ is  $1$.

So this example fits our hypotheses:  an \etale\ groupoid $ G $ with closed orbits, abelian and continuously varying stability subgroups such that $\Rr$ is \etale\ with  finite dynamic asymptotic dimension and such that $\widehat{\Aa} $ has finite topological dimension. In particular, \Cref{non-principal groupoids} gives 
\[
\dim_\nuc(\Cst( G ))\leq (1+1)(0+1)-1\leq 1.
\]

Here  $\Cst( G )\cong C_0(X)\rtimes\ZZ$ has continuous trace by \cite[Theorem~5.1]{Williams:cts_trace_transformation_groups}.  Thus
\[\dim_\nuc(\Cst( G ))=\dim\big(\Cst( G )^\wedge\big)=\dim\big(\Prim(C_0(X)\rtimes\ZZ)\big),\] and we could have used the description of the primitive ideal space from \cite[Theorem~5.3]{Williams:prim} to determine the topological dimension of $\Prim(C_0(X)\rtimes\ZZ)$. 
\end{example}

Our second example, \Cref{ex-CaH} below, is a $\Cst$-algebra of a graph groupoid.  We start with the required background on directed graphs.

Let  $E=(E^0,E^1,r,s)$  be a row-finite directed graph, so that $r\inv (v)$ is finite for every $v\in E^0$. A finite path is a finite sequence $\mu=\mu_{1}\mu_{2}\cdots\mu_k$ of edges $\mu_i\in E^1$ with $s(\mu_j)=r(\mu_{j+1})$ for $1\le j\le k-1$; write $s(\mu)=s(\mu_k)$ and $r(\mu)=r(\mu_{1})$, and call $|\mu|\coloneqq k$ the  length of $\mu$. An  infinite path $x=x_{1}x_{2}\cdots$ is defined similarly; note that $s(x)$ is undefined. Let $E^*$ and $E^\infty$ denote the set of all finite paths and infinite paths in $E$, respectively. If $\mu=\mu_{1}\cdots\mu_k$ and $\nu=\nu_{1}\cdots\nu_j$ are finite paths with $s(\mu)=r(\nu)$, then $\mu\nu$ is the path $\mu_{1}\cdots\mu_k\nu_{1}\cdots\nu_j$. When $x\in E^\infty$ with $s(\mu)=r(x)$ define $\mu x$ similarly. A return path is a finite path $\mu$ of non-zero length such that $s(\mu)=r(\mu)$.  By \cite[Corollary~2.2]{KumjianPaskRaeburnRenault}, the cylinder sets
\[
Z(\mu)\coloneqq\{x\in E^\infty: x_{1}=\mu_{1},\ldots,x_{|\mu|}=\mu_{|\mu|}\},
\]
parameterized by $\mu\in E^*$, form a basis of compact, open sets for a locally compact,   totally disconnected, Hausdorff topology on $E^\infty$.

Given a row-finite graph $E$, in  \cite{KumjianPaskRaeburnRenault}, Kumjian, Pask, Raeburn and Renault define a groupoid $G_E$ as follows. Two paths $x,y\in E^\infty$ are shift equivalent with lag $k\in\ZZ$ (written $x\sim_k y$) if there exists $N\in\NN$ such that $x_i=y_{i+k}$ for all $i\ge N$. Then the groupoid is
\[ G_E\coloneqq\{(x,k,y)\in E^\infty\times \ZZ\times E^\infty  : x\sim_k y\}.\]
with composable pairs
\[
G_E\comp \coloneqq\{\big((x,k,y),(y,l,z)\big):(x,k,y),(y,l,z)\in G_E\},
\]
and composition and inverse given by
\[
(x,k,y)\cdot (y,l,z)\coloneqq (x,k+l,z)\quad\text{and}\quad(x,k,y)\inv \coloneqq (y,-k,x).
\]
For  $\mu,\nu\in E^*$ with $s(\mu)=s(\nu)$, let $Z(\mu,\nu)$ be the set
\[
\{(x,k,y) : x\in Z(\mu), y\in Z(\nu), k=|\nu|-|\mu|, x_i=y_{i+k}\text{ for } i>|\mu|\}.
\]
By \cite[Proposition~2.6]{KumjianPaskRaeburnRenault}, the collection of sets
\[
\{Z(\mu,\nu) : \mu,\nu\in E^*, s(\mu)=s(\nu)\}
\]
is a basis of compact, open sets for a second-countable, locally compact, Hausdorff topology on $G_E$ such that  $G_E$ is an \etale\ groupoid.  After identifying $(x,0,x)\in G_E\z $ with $x\in E^\infty $, \cite[Proposition~2.6]{KumjianPaskRaeburnRenault} says that the topology on $G_E\z $ is identical to the topology on $E^\infty $.

To discuss \Cref{ex-CaH} below we  will need item~\ref{lemma-directed-graph-groupoids-1} of the  following  lemma which was proved by the third author and Lisa Orloff Clark.

\begin{lemma}\label{lemma-directed-graph-groupoids} Let $E$ be a row-finite directed graph with no sources. 
\begin{enumerate}[label=\textup{(\arabic*)}]
    \item\label{lemma-directed-graph-groupoids-1}  If $E$ has no return paths, then 
    $\dad(G_E)=0$. 
    \item\label{lemma-directed-graph-groupoids-2}  If $E$ has a return path, then 
    $\dad(G_E)=\infty$.
\end{enumerate}
\end{lemma}

\begin{proof}For~\ref{lemma-directed-graph-groupoids-1}, suppose that $E$ has no return paths. Fix an open, precompact  subset $K$ of $G_E$. Without loss of generality, we can replace $K$ by $K\cup K\inv $.  Let $U_0=E^\infty$. There exists a finite subset $F$ of the fibered product $E^*\times_s E^*\coloneqq\{(\mu,\nu)\in E^*\times E^*: s(\mu)=s(\nu)\}$ 
such that $K\subset \cup_{(\mu,\nu)\in F} Z(\mu,\nu)$. Without loss of generality, we may assume that $(\mu,\nu)\in F$ implies $(\nu,\mu)\in F$. For a path $\mu$, write $V(\mu)$ for the subset of vertices on $\mu$. Set $V_F=\cup\{V(\mu) : (\mu,\nu)\in F\}$ and let $E_F$ be the subgraph of $E$ of paths with vertices in $V_F$.

Since there are no return paths in $E$, the set $\{(\alpha,\beta)\in E_F^*\times_s E_F^*\}$ is finite.  We claim that the groupoid 
\[
G_K=\bigcup_{n\in\NN}\{k_{1}\cdots k_n : k_i\in K\text{ for } 1\leq i\leq n\}
\]
generated by $K$ is contained in the compact subset
\begin{equation}\label{eq-union}
\bigcup_{\{(\alpha,\beta)\in E_F^*\times_s E_F^*\}}Z(\alpha,\beta)
\end{equation}
of $G_E$.  We prove the claim by induction on the number $n$ of products.  Let $n\geq 1$ and assume that any product $k_{1}\cdots k_n$ with $k_i\in K$ is in the union at \eqref{eq-union}. Now consider a product $l_{1}\cdots l_i\cdots l_{n+1}$  with $l_i\in K$. By the induction hypothesis there exist $(\eta,\zeta)\in E_F^*\times_s E_F^*$ and $x\in E^\infty$ such that $l_{1}\cdots l_n=(\eta x, |\eta|-|\zeta|, \zeta x)$. Since $l_{n+1}\in K$, there exists $(\mu, \nu)\in F$ such that $l_{n+1}\in Z(\mu,\nu)$. Thus, $\mu, \nu\in E_F^*$  and there exists $y\in E^\infty$ such that $l_{n+1}=(\mu y, |\mu|-|\nu|, \nu y)$. We have \[l_{1}\cdots l_nl_{n+1}=(\eta x, |\eta|-|\zeta|, \zeta x)(\mu y, |\mu|-|\nu|, \nu y).\]
Thus, $\zeta x=\mu y$.  Suppose that $|\mu|\geq |\zeta|$. Then there exists $\mu'\in E_F^*$ such that $\mu=\zeta\mu'$ and $x=\mu'y$. 
Now
\begin{align*}
(\eta x, |\eta|-|\zeta|, \zeta x)(\mu y, |\mu|-|\nu|, \nu y)&=(\eta\mu' y, |\eta|-|\zeta|, \mu y)(\mu y, |\mu|-|\nu|, \nu y)\\
&=(\eta\mu' y, |\mu|-|\nu|+|\eta|-|\zeta|, \nu y).
\end{align*}
Now notice that both $\eta\mu'$ and $\nu$ are in $E_F^*$. Thus, $l_{1}\cdots l_nl_{n+1}$ is in the union at \eqref{eq-union},  which proves the claim. This shows that $\dad(G_E)=0$. 

\medskip

For~\ref{lemma-directed-graph-groupoids-2}, suppose that $E$ has a return path $\mu$.  We argue by contradiction: suppose that $G_E$ has finite dynamic asymptotic dimension $d$. The length $|\mu|$ of $\mu$ is at least 1. We write $\mu^\infty$ for the infinite path obtained by concatenating $\mu$ with itself infinitely many times. Then $(\mu^\infty, |\mu|,\mu^\infty)\in G_E$.  Let $K$ be an open, precompact subset of $G_E$ containing $(\mu^\infty, |\mu|,\mu^\infty)$. There exist open neighborhoods $U_0,\dots, U_d$ which cover $r(K)\cup s(K)$ and the groupoid $G_i$ generated  by $\{g\in K : s(g)=r(r)\in U_i\}$ is  precompact. 
Without loss of generality,  $\mu^\infty\in U_0$. Then for each $n\geq 1$ the composition $(\mu^\infty, n|\mu|,\mu^\infty)$ of $(\mu^\infty, |\mu|,\mu^\infty)$ with itself $n$ times is in $G_0$.
But $\{(\mu^\infty, n|\mu|,\mu^\infty)\}_{n\geq 1}$ has no convergent subsequence, which is a contradiction. 
\end{proof}

\begin{example}\label{ex-CaH}  Let $E$ be the graph of Figure~\ref{graph E} which was considered in \cite[Example~1 in \S7]{Clark-anHuef-RepTh}:
\begin{figure}[!htb]
\[
\xymatrix{\overset{v_0}\bullet&&\overset{v_{1}}\bullet\ar[ll]&&\overset{v_{2}}\bullet\ar[ll]&&\dots\ar[ll]\\
\overset{w_{0,0}}\bullet\ar@(dr,dl)^{e_{0,0}}\ar@/^/[u]^{f_{0,0}}\ar@/_/[u]_{f_{0,1}}
&&\overset{w_{1,0}}\bullet\ar@/^/[u]^{f_{1,0}}\ar@/_/[u]_{f_{1,1}}\ar@/_/[d]_{e_{1,1}}
&&\overset{w_{2,0}}\bullet\ar@/^/[u]^{{f}_{2,0}}\ar@/_/[u]_{{f}_{2,1}}
\ar@/_/[ld]_{e_{2,2}}&&\dots\\
&&\underset{w_{1,1}}\bullet\ar@/_/[u]_{e_{1,0}}&\underset{w_{2,1}}\bullet
\ar@/_/[rr]_{e_{2,1}}&&\underset{w_{2,2}}\bullet\ar@/_/[lu]_{e_{2,0}}&\dots
}
\]
\caption{The graph $E$}\label{graph E}
\end{figure}
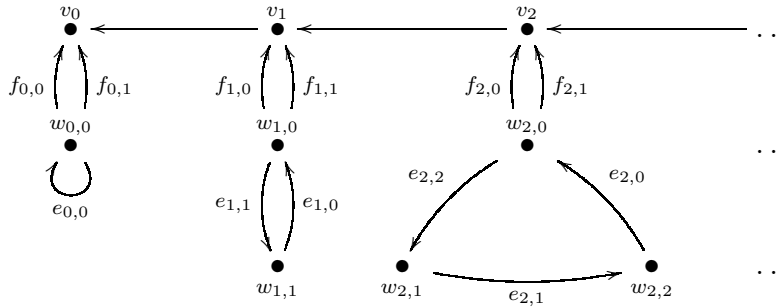
We write  $ G_E$ for the graph groupoid and identify its $\Cst$-algebra  with the $\Cst$-algebra $\Cst(E)$ of the graph.   
We know a lot about $\Cst(G_E)$: it has bounded trace but is not a Fell algebra by \cite[p.\ 188]{Clark-anHuef-RepTh}, hence does not have continuous trace.  Since $E$ has return paths, $\Cst( G_E)$ is not AF by \cite[Theorem 2.4]{KumjianPaskRaeburn-graphs}, and since the return paths do not have entries, $\Cst( G_E)$ is AF-embeddable by \cite[Theorem 1.1]{Schafhauser-AFE}.  Here we will use  \Cref{non-principal groupoids}  to show that $\Cst( G_E)$ has nuclear dimension $1$. 

For $x\in E^{\infty}$, we write  $\Aa_x=\{ g\in G_E : r(g)=x=s(g) \}$ for the stability subgroup at $x$. 
Let $\alpha^{k,n}$ be a return path of length $n+1$ with range and source some $w_{n,k}$. If  $x_{\alpha^{k,n}}=\alpha^{k,n}\alpha^{k,n}\ldots $ then the stability subgroup at $x_{\alpha^{k,n}}$ is isomorphic to  $(n+1)\ZZ$; all other stability subgroups are trivial.  
The stability subgroups vary continuously and the
orbits are closed.

We first claim that the topological dimension of $\widehat{\Aa} $ is $1$. We write $\widehat{\Aa}_x$ for the dual of the stability subgroup $\Aa_x$. Since  $E^\infty$ is  a countable set,
\[
\widehat{\Aa} =\bigcup_{x\in E^\infty} \widehat{\Aa}_{x}\times\{x\}
\]
is a countable union of closed subsets. Each $\widehat{\Aa}_{x}\times\{x\}$ is either $\{(0,x)\}$ or homeomorphic to $\TT$. Since $\TT$ has topological dimension $1$,  the countable-union theorem \cite[Chapter 2, Theorem~2.5]{Pears:dim} or \cite[Theorem~1.7.1]{Coornaert} gives $\dim(\widehat{\Aa} )=1$, as claimed.  
\goodbreak

Second, we claim that the dynamic asymptotic dimension of $\Rr\coloneqq G_E/\Aa$ is zero.
To see this, let $F$ be the graph shown in Figure~\ref{graph F}.  
\begin{figure}[!htb]
\[\xymatrix{\overset{v'_0}\bullet&&\overset{v'_{1}}\bullet\ar[ll]&&\overset{v'_{2}}\bullet\ar[ll]&&\dots\ar[ll]\\
\overset{w'_{0,0}}\bullet\ar@/^/[u]^{f'_{0,0}}\ar@/_/[u]_{f'_{0,1}}
&&\overset{w'_{1,0}}\bullet\ar@/^/[u]^{f'_{1,0}}\ar@/_/[u]_{f'_{1,1}}&&\overset{w'_{2,0}}\bullet
\ar@/^/[u]^{f'_{2,0}}\ar@/_/[u]_{f'_{2,1}}&&\dots\\
\bullet\ar[u]^{e'_{0,0}}&&\overset{w'_{1,1}}\bullet\ar[u]^{e'_{1,0}}&&\overset{w'_{2,1}}\bullet
\ar[u]^{e'_{2,0}}&&\dots\\
\bullet\ar[u] &&\bullet\ar[u]^{e'_{1,1}}&&\overset{w'_{2,2}}\bullet
\ar[u]^{e'_{2,1}}&&\dots\\
\vdots\ar[u]&&\vdots\ar[u] &&\vdots\ar[u]^{e'_{2,2}}}\] 
\caption{The graph $F$}\label{graph F}
\end{figure}
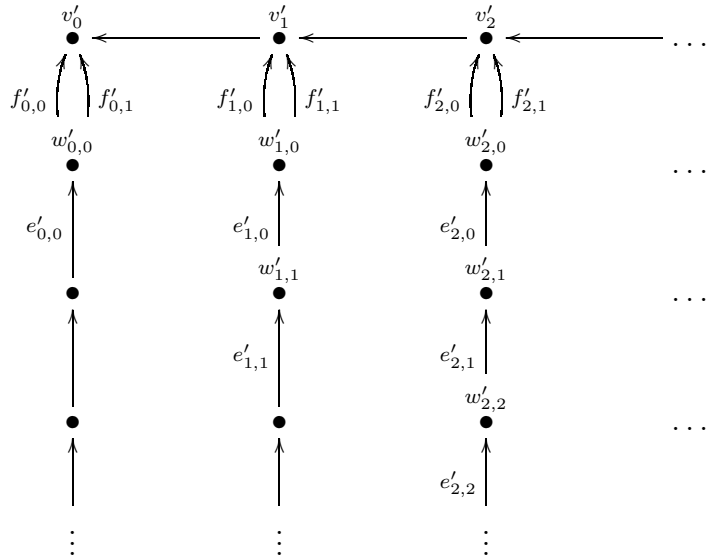

 \noindent Consider the subset  \[
 U=\bigcup_{i\geq 0}\big(Z(v'_i)\cup_{j\leq i} Z(w'_{i,j})\big)
 \] 
 of $F^\infty$. Notice that $U$ is both open and closed. (But $U$ is not an invariant subset. To see this, let $x$ be the infinite path ending with the edge  $f'_{0,0}$. Let $\sigma$ be the shift map on $E^\infty$.  Then $(x, 2,\sigma^2(x))\in G_F$, 
 its range $x$ is in $U$ but its source $\sigma^2(x)$ is not in $U$.) 

By \cite[p.~189]{Clark-anHuef-RepTh}, $\Rr$ is isomorphic to the reduction $(G_F)|_U$ to $U$ of  $G_F$. In particular, $\Rr$ is \etale\ because $U$ is open and because $G_{F}$ is \etale.
Since $F$ is row-finite with no sources or return paths, $G_F$  has dynamic asymptotic dimension  $0$ by \Cref{lemma-directed-graph-groupoids}.  Since $U$ is both open and closed, it follows that the dynamic asymptotic dimension of $(G_F)|_U$ is also $0$ by \Cref{lem:dad of restriction'}. Thus, $\dad(\Rr)=0$ 
as well. Now \Cref{non-principal groupoids} applies to give $\dim_\nuc(\Cst(E))\leq 2\cdot 1-1=1$; we know it is at least $1$ since $\Cst(E)$ is not AF.

\end{example}

\appendix

\section{Notation Index}
\[\begin{array}{cl}
    G\z &  \text{unit space of } G\\
    r\colon G\to G\z & \text{range map} \\
    s\colon G\to G\z & \text{source map}\\
    G\comp  & \text{composable pairs in } G\\
    \langle X\rangle & \text{subgroupoid generated by } X\\
    \restrict{G}{W} &
\{e\in G : s(e), r(e)\in W\subset G\z \}\\
{[x]}& \text{orbit of } x\in G\z \\
xG & \{e\in G : r(e)=x\}\\
Gx & \{e\in G : s(e)=x\}\\
C_c(G) & \text{compactly supported } \CC\text{-valued functions on } G\\
\sigma^{x} & \text{measure in left Haar system for } x\in G\z \\
\sigma_x(U) & \sigma^x(U\inv )\\
\Cst(G,\sigma) & \text{full } \Cst\text{-algebra of } G \text{ with respect to Haar system } \sigma\\
\Cst_\red(G,\sigma) & \text{reduced } \Cst\text{-algebra of } G \text{ with respect to Haar system } \sigma\\
(E, \iota, \pi)& \text{ a twist $G\z\times\TT \overset{\iota}{\to} E\overset{\pi}{\to} G$, see \Cref{def:twist2}.}\\
C_c(E;G)& \text{$f\in C_c(E)$ such that $f(z\cdot e)=zf(e)$ for  $z\in\TT$ and $e\in E$}\\ & \text{for a twist  $(E, \iota, \pi)$  over  $G$} \\
\Cst(E; G) & \text{full  twisted groupoid } \Cst\text{-algebra for a twist  $(E, \iota, \pi)$  over  $G$}\\
\Cst_\red(E; G) & \text{reduced twisted groupoid } \Cst\text{-algebra for a twist  $(E, \iota, \pi)$  over  $G$} \\
C_{K}(E;G) & \{f\in C_c(E;G):\supp f\subset K \} \text{ where } K\subset E\\
X_+ & \text{Alexandrov one-point-compactification of } X\\
\widetilde{A} & \text{minimal unitization of a } \Cst\text{-algebra } A\\
\cG & \text{Alexandrov groupoid for a groupoid } G \text{, see \Cref{lem:etale:tilde G:topologically'}}\\
(\widetilde{E},\widetilde{\iota},\widetilde{\pi}) & \text{Alexandrov twist for a twist } (E,\iota,\pi)
    \text{, see \Cref{tilde E' v2}}\\
\text{[}f,h\text{]} & \text{commutator, as in Equation } \eqref{eq: comm}\\
\end{array}\]

\printbibliography

\end{document}